\documentclass[a4paper,12pt]{amsart}
\usepackage{xcolor}

\title[Existentially closed pmp actions of free groups]{Existentially closed measure-preserving actions of free groups}

\author[A. Berenstein]{Alexander Berenstein}
\address{Alexander Berenstein\\
Universidad de los Andes,
Cra 1 No 18A-12, Bogot\'{a}, Colombia.}
\urladdr{https://pentagono.uniandes.edu.co/~aberenst/}

\author[C. W. Henson]{C. Ward Henson}
\address{C. Ward Henson
\\ University of Illinois, Urbana-Champaign \\ 
1409 West   Green Street \\Urbana, IL 61801 \\USA.}
\urladdr{https://math.illinois.edu/directory/profile/cwhenson}

\author[T. Ibarluc\'ia]{Tom\'as Ibarluc\'ia}
\address{Tom\'as Ibarluc\'ia
\\ Universit\'e Paris Cit\'e \\
  CNRS \\
  IMJ-PRG \\
  F-75006 Paris \\
  France.}
\urladdr{https://webusers.imj-prg.fr/~tomas.ibarlucia/}

\thanks{Research partially supported by ECOS-Nord project 048-2019 \emph{Teor\'ia de modelos y grupos polacos} (first author), Simons Foundation collaboration grants 202251 and 422088 (second author), and ANR contract AGRUME (ANR-17-CE40-0026) (third author).}

\swapnumbers

\usepackage{amsmath, amssymb, amsthm, amscd, mathtools, enumitem}

\usepackage[T1]{fontenc}
\usepackage{lmodern}

\usepackage{setspace}
\usepackage{geometry}
\usepackage{stmaryrd}

\makeatletter
\g@addto@macro\bfseries{\boldmath}
\makeatother

\usepackage[hidelinks]{hyperref}

\def\Ind#1#2{#1\setbox0=\hbox{$#1x$}\kern\wd0\hbox to 0pt{\hss$#1\mid$\hss}
\lower.9\ht0\hbox to 0pt{\hss$#1\smile$\hss}\kern\wd0}
\def\ind{\mathop{\mathpalette\Ind{}}}
\def\Notind#1#2{#1\setbox0=\hbox{$#1x$}\kern\wd0\hbox to 0pt{\mathchardef
\nn="3236\hss$#1\nn$\kern1.4\wd0\hss}\hbox to 0pt{\hss$#1\mid$\hss}\lower.9\ht0
\hbox to 0pt{\hss$#1\smile$\hss}\kern\wd0}

\usepackage{graphicx}
\makeatletter
\newcommand{\bigperp}{%
  \mathop{\mathpalette\bigp@rp\relax}%
  \displaylimits
}
\newcommand{\bigp@rp}[2]{%
  \vcenter{
    \m@th\hbox{\scalebox{\ifx#1\displaystyle1.7\else1.4\fi}{$#1\perp$}}
  }%
}
\makeatother

\def\mbN{\mathbb{N}}
\def\mbZ{\mathbb{Z}}
\def\mbQ{\mathbb{Q}}
\def\mbF{\mathbb{F}}
\def\mbR{\mathbb{R}}

\def\mbI{\mathbb{I}}

\def\mcE{\mathcal{E}}
\def\mcL{\mathcal{L}}
\def\mcU{\mathcal{U}}

\def\mcX{\mathcal{X}}
\def\mcY{\mathcal{Y}}
\def\mcZ{\mathcal{Z}}

\def\MBAQ{\mathcal{MBA}_\mbQ}

\def\APA{{\rm APA}}
\def\ARV{{\rm ARV}}
\def\PMPk{{\rm PMP$_k$}}
\def\PMPFk*{{\rm PMP$_k^*$}}
\def\PMPF1*{{\rm PMP$_1^*$}}
\def\mathPMPFk*{{\rm PMP}_k^*}
\def\wFk{\widehat{\mbF}_k}

\DeclareMathOperator{\Aut}{Aut}
\DeclareMathOperator{\Sym}{Sym}

\DeclareMathOperator{\Clop}{Clop}
\DeclareMathOperator{\acl}{acl}

\DeclareMathOperator{\tp}{tp}
\DeclareMathOperator{\stp}{stp}
\DeclareMathOperator{\Th}{Th}

\newcommand{\meq}{\mathrm{meq}}
\newcommand{\inv}{^{-1}}
\newcommand{\ov}[1]{\overline{#1}}

\newcommand{\eqstp}{\stackrel{\boldsymbol{.}}{\equiv}}
\newcommand{\actson}{\curvearrowright}

\theoremstyle{plain}        \newtheorem{fact}{Fact}[section]
\theoremstyle{plain}        \newtheorem{theorem}[fact]{Theorem}
\theoremstyle{plain}        \newtheorem*{theorem*}{Theorem}
\theoremstyle{plain}        \newtheorem{lem}[fact]{Lemma}
\theoremstyle{plain}        \newtheorem{prop}[fact]{Proposition}
\theoremstyle{plain}        \newtheorem{cor}[fact]{Corollary}
\theoremstyle{definition}   \newtheorem{rem}[fact]{Remark} 
\theoremstyle{definition}   \newtheorem{defin}[fact]{Definition}
\theoremstyle{definition}   
\theoremstyle{definition}   \newtheorem{assumption}[fact]{Assumption}
\theoremstyle{definition}   
\theoremstyle{definition}   
\theoremstyle{definition}   \newtheorem{question}[fact]{Question}
\theoremstyle{definition}   \newtheorem*{question*}{Question}
\theoremstyle{definition}   \newtheorem{example}[fact]{Example}
\theoremstyle{definition}   \newtheorem{examples}[fact]{Examples}

\begin{document}

\begin{abstract}
This paper is motivated by the study of probability measure-preserving (\emph{pmp}) actions of free groups using continuous model theory.  Such an action is treated as a metric structure that consists of the measure algebra of the probability measure space expanded by a family of its automorphisms. We prove that the existentially closed pmp actions of a given free group form an elementary class, and therefore the theory of pmp $\mbF_k$-actions has a model companion.  We show this model companion is stable and has quantifier elimination. We also prove that the action of $\mbF_k$ on its profinite completion with the Haar measure is metrically generic and therefore, as we show, it is existentially closed.

We deduce our main result from a more general theorem, which gives a set of sufficient conditions for the existence of a model companion for the theory of $\mbF_k$-actions on a separably categorical, stable metric structure.
\end{abstract}

\maketitle
\setcounter{tocdepth}{1}
\tableofcontents

\section{Introduction}

One of the key notions that model theory abstracted from the theory of fields and applied successfully in many other algebraic settings is that of \emph{existential closedness}. Existentially closed structures generalize algebraically closed fields, and in appropriate circumstances---for example, when the base theory is inductive and Robinson's \emph{model companion} exists---this provides a model-theoretic concept of \emph{generic} structure, complementary to other mathematical notions of genericity.

Continuous logic makes it possible to consider these ideas in more analytic contexts, and in particular in ergodic theory and measured group theory. Consider a factor $\mcX$ of a probability measure-preserving (\emph{pmp}) system $\mcY$ of some countable group $\Gamma$; we think of pmp systems in terms of their associated measure algebras, so we see a factor as an inclusion $\mcX\subseteq \mcY$.  Then we say $\mcX$ is \emph{existentially closed in $\mcY$} if for every $\epsilon>0$, every finitely many measurable sets $\{a_i\}_{i\in I}$ in $\mcX$ and $\{b_j\}_{j\in J}$ in $\mcY$, and every finitely many group elements $\{\gamma_l\}_{l\in L}$ from $\Gamma$, we can find measurable sets $\{c_j\}_{j\in J}$ in $\mcX$ such that
\[|\mu(a_i\cap b_j\cap \gamma_l b_k) - \mu(a_i\cap c_j\cap \gamma_l c_k)|<\epsilon\]
for every $i\in I$, $j,k\in J$, and $l\in L$. This is a strengthening of saying that $\mcY$ is \emph{weakly contained} in its factor 
$\mcX$, in the sense of Kechris \cite{kechrisGlobal} (which, as is well known, also has a natural model-theoretic phrasing, cf.\ \cite[Lemma~2.6]{ibatsaStrong}). The reader may note that $\mcX$ is existentially closed in $\mcY$ if and only if there is an ultrapower $\mcX^\mcU$ of $\mcX$ such that $\mcX\subseteq \mcY\subseteq \mcX^\mcU$, where the composition of the inclusions is the canonical diagonal inclusion into the ultrapower. We say then that a pmp $\Gamma$-system $\mcX$ is \emph{existentially closed (e.c.)} if for every extension $\mcX\subseteq\mcY$, we have that $\mcX$ is existentially closed in 
$\mcY$.

These definitions do not bring anything new into classical ergodic theory, that is, to the study of $\mbZ$-systems, since every aperiodic $\mbZ$-system is existentially closed, and conversely \cite{bbhu08}. In fact, for amenable groups generally, the model theory of pmp systems is completely coded by---and codes---the IRSs (= invariant random subgroups) of the systems; for instance, if $\mcX\subseteq\mcY$ are atomless $\Gamma$-systems and $\Gamma$ is amenable, then $\mcX$ is existentially closed in $\mcY$ iff $\mcX$ and $\mcY$ have the same IRS: see Giraud's work \cite{girHyperfinite}. In particular, the e.c.\ pmp systems of an amenable group are given precisely by its free actions on atomless spaces. See some further comments and references at the end of Subsection \ref{ss:meas-alg} below.

For non-amenable groups, on the other hand, the properties defined above are largely unexplored. The importance acquired by the notion of weak containment in recent years, however, is an indication that the study of existentially closed factors and systems of non-amenable groups should be worthwhile and fertile.\footnote{In a different but not unrelated vein, a concrete demonstration of the usefulness of the model-theoretic viewpoint in measured group theory was given in \cite{ibatsaStrong}, where the notion of existential algebraic closure shed new light on certain rigidity phenomena of strongly ergodic systems.}

In this paper, we make a first step in this direction by showing that the class of existentially closed pmp systems of a given non-abelian free group $\mbF_k$ is axiomatizable in continuous logic and is well-behaved. In model-theoretic terminology:

\begin{theorem}[See Theorem \ref{thm:main-example}]
The theory of pmp $\mbF_k$-systems has a model completion, which has quantifier elimination and is stable.

Moreover, in the model completion, the algebraic closure of a set $A$ equals the subsystem generated by $A$, and the stable independence relation is given by the usual probabilistic independence relation between the generated subsystems.
\end{theorem}

Of independent interest is the following characterization of the e.c.\ pmp $\mbF_k$-systems, on which the continuous logic axiomatization relies, and which stems from ideas native to the study of Robinson generic automorphisms of fields, as we explain below. For clarity we set $k=2$, that is, we consider probability spaces $(X,\mu)$ endowed with two automorphisms $\tau_1,\tau_2$. If $x,y,x',y'$ are finite random variables on $X$ (i.e., with finite range), we write $xy\equiv x'y'$ to say that $xy$ and $x'y'$ have the same joint distribution, and we write $x\ind_{x'}yy'$ to say that $x$ and the joint random variable $yy'$ are conditionally independent given $x'$. We also write $\tau_i(x)$ for the random variable $x\circ\tau_i\inv$.
 
We state this characterization next in an informal way, expressed using approximate versions of the standard concepts from probability that are defined above.

\begin{theorem}[See Definition \ref{d:class of ec models} and Section \ref{s:particular-cases}] A pmp system $\mcX=(X,\mu,\tau_1,\tau_2)$ is existentially closed iff it is atomless and the following condition holds: whenever $x,y_0,y_1,y_2$ are finite random variables on $X$ such that, \emph{up to a given small error},
$$xy_0\equiv \tau_1(x)y_1 \equiv\tau_2(x)y_2,\ y_0\ind_{x}\tau_1(x)\tau_2(x),\ y_1\ind_{\tau_1(x)}x\tau_2(x),\text{ and }y_2\ind_{\tau_2(x)}x\tau_1(x),$$
we can find a finite random variable $z$ on $X$ such that, \emph{up to a corresponding small error},
$$xy_0y_1y_2\equiv xz\tau_1(z)\tau_2(z).$$
\end{theorem}

As we will show, the scheme above works for characterizing e.c.\ $\mbF_2$-actions in much greater generality, namely, for actions on models of any given $\aleph_0$-categorical, exact, superstable metric theory---of which atomless probability algebras are an example\footnote{By \emph{$\mbF_k$-action} we will usually mean a metric structure expanded with $k$ automorphisms. In the case of probability measure algebras, we may as above also use the term \emph{pmp $\mbF_k$-system}, especially when the action is induced by a concrete measure-preserving action $\mbF_k\actson (X,\mu)$ on a probability space.}. The scheme is not original; in this form it is due to Anand Pillay, up to the error terms that we have added to make it work in the metric setting. It is based on model-theoretic ideas that date back to the works of Lascar \cite{lasBeaux} and Chatzidakis--Hrushovski \cite{chahru} on structures (especially, fields) endowed with a Robinson generic automorphism.\footnote{See also Chatzidakis--Pillay \cite{chapil} and Kikyo--Pillay \cite{kikpil}; the latter pointed out that these ideas can be generalized from one to several automorphisms.} The reader may compare this scheme with the axioms of ACFA \cite[\textsection 1.1]{chahru} (cf.\ also the simplified version in \cite[\textsection 1.3]{chatziSurvey}), and with the generalizations for strongly minimal sets given by Kikyo--Pillay \cite[Fact~1.5, Lem.~2.1]{kikpil}. The first named author of the present paper learned this particular form in private communication from Pillay in the early 2000s, for the axiomatization of existentially closed 
$\mbF_2$-actions on classical (i.e., discrete) $\aleph_0$-categorical, superstable structures. This result, however, seems to have never been published,\footnote{The seminal work of Lascar addresses the case of classical $\aleph_0$-categorical, superstable structures with a single automorphism: cf.\ \cite[Exemple~4]{lasBeaux}.} so we remedy this and extend it at the same time to the metric setting.

Going from the preceding characterization of the e.c.\ models to their axiomatization in continuous logic requires some work---whereas it is obvious in the classical first-order ($\aleph_0$-categorical) case. The naive way of proceeding, which is the one we follow, requires a strong form of definability of the stable independence relation (and of equality of strong types over finite sets), which does not hold in every 
$\aleph_0$-categorical, superstable metric theory. We analyze these technical issues in 
Section~\ref{s:def-of-indep}. We thus identify a set of general assumptions on the base theory that allow us to formalize the previous scheme in continuous logic, and with them we prove our main general theorem in Section~\ref{s:main}. As we state next and prove later in Section~\ref{s:particular-cases}, a convenient sufficient condition to satisfy these assumptions is a strong form of homogeneity that we introduce in Definition \ref{D:equi-homogeneous}.

\begin{theorem}[See Theorem \ref{thm:main}  and Lemma \ref{l:from-homog-to-cont-sym}]
Let $T$ be a model-complete, $\aleph_0$-categorical, superstable metric theory whose separable model is \emph{equi-homogeneous}. Then the theory of $\mbF_k$-actions on models of $T$ has a model completion.
\end{theorem}

This result applies to probability algebras---our main motivation---but also, more generally, to randomizations of classical structures. On the other hand, our assumptions do not hold in the basic case of Hilbert spaces, for which it is known that the e.c.\ $\Gamma$-actions of any countable group $\Gamma$ are axiomatizable \cite{berHilb}---and in a particularly simple manner, namely, by their existential theory. Also, in the recent preprint \cite{sciBanach}, A.\ Scielzo shows that the theory of atomless $L^p$ Banach lattices with an automorphism has a model completion---an example in which the base separable structure is not even homogeneous. It thus seems plausible that an alternative argument could yield a positive answer to the following.

\begin{question}
Does the theory of $\mbF_k$-actions on models of $T$ admit a model completion for every model complete, $\aleph_0$-categorical, superstable metric theory $T$?
\end{question}

In Section~\ref{s:metric-generics}, we record a fact relating Robinson genericity in continuous logic to the notion of \emph{metrically generic} actions from the work of Ben Yaacov--Berenstein--Melleray \cite{BBM-topometric}. More precisely, we give a proof that metrically generic actions are existentially closed. Later, in Section~\ref{s:profinite-completion}, we use and adapt the methods from \cite{BBM-topometric} to give two examples of metrically generic pmp $\mbF_k$-actions. In particular, strengthening a result of Kechris \cite{kechrisWeak}, we show:

\begin{theorem}[See Theorem \ref{thm:metric-gen-profinite}]
The pmp system of $\mbF_k$ on its profinite completion equipped with the Haar measure is metrically generic, and thus existentially closed.
\end{theorem}

An interesting question that arises is whether every separable e.c.\ pmp $\mbF_2$-action is metrically generic (see the end of Subsection \ref{ss:meas-alg}), or equivalently:

\begin{question}
Are all separable e.c.\ pmp $\mbF_2$-actions approximately isomorphic?
\end{question}

Throughout the paper, we will concentrate mostly on free groups $\mbF_k$ with $k$ finite, and some details are given only in the case $k=2$.  However, our main results from Sections \ref{s:main} and \ref{s:particular-cases} also hold when $k=\omega$, as we show. It seems likely that the same is true for the results in Sections \ref{s:metric-generics} and \ref{s:profinite-completion}, with the appropriate corresponding notion of metric generics, but we leave those details to the interested reader.

The paper is written in the language of model theory, and familiarity with continuous logic is assumed. For basic background see \cite{bbhu08} and \cite{benusv10}. For an extensive treatment of the model theory of probability spaces, see \cite{bh23}.
We hope however that the introductory explanations given above for the case of pmp systems will help the interested readers coming from measured group theory to make the necessary translations.

\noindent\hrulefill

\section{Definability of independence}\label{s:def-of-indep}

Throughout this section and Section 3, unless otherwise specified, we let $\mcL$ be a countable metric first-order language and let $T$ be an $\aleph_0$-categorical $\mcL$-theory.  Thus $T$ has a unique separable model up to isomorphism.  We let $M$ be this separable model; following the standard notation, we write $d$ for its metric. To simplify the presentation we will assume that the language of $T$ is one-sorted, but everything below can be adapted to the multi-sorted case by replacing the finite powers $M^n$ with finite products of sorts of~$M$. We endow the finite powers $M^n$ with the induced distance $d(a,b)=\max_{i<n}d(a_i,b_i)$. We denote the automorphism group $\Aut(M)$ of $M$ by $G$. The group $G$ acts isometrically on the powers $M^n$ by the diagonal action: $ga\coloneqq (ga_i)_{i<n}$.

In this situation, an important consequence of the Ryll-Nardzewski Theorem (see \cite[Fact 1.14]{benusvD-fin}) is the following characterization of definable predicates and sets:  (a) $P \colon M^n \to [0,1]$ is a definable predicate for $T$ if and only if it is continuous and $G$-invariant; (b) $D \subseteq M^n$ is a definable set if and only if it is closed and $G$-invariant.  (For a proof of (a), see \cite[Prop.~2.2]{benkai13}; (b) is proved by applying (a) to $P(x) \coloneqq \text{dist}(x,D)$.)  In particular, whenever $D$ is a closed, $G$-invariant set, we may use the quantifiers $\sup_{x \in D}$ and $\inf_{x \in D}$ without leaving the expressive capabilities of the language of $T$.

Note that when we use the concept \emph{definable} we always mean \emph{without parameters}.  If parameters are allowed, they will always be mentioned specifically.

In what follows we will assume in fact that $T$ satisfies a strengthened form of $\aleph_0$-categoricity that we define next, which always holds in the classical, discrete first-order setting.  

\begin{defin}
Let $T$ and $M$ be as above.  We will say that the $\aleph_0$-categorical theory $T$ is \emph{exact} if for every $n\in\mbN$ and $a\in M^n$, the theory of the structure $(M,a)$ is $\aleph_0$-categorical.
\end{defin}

In other words, we make the hypothesis that categoricity is preserved by naming finite tuples of elements. There are several examples known to satisfy this condition, among them Hilbert spaces, atomless probability algebras and randomizations of discrete $\aleph_0$-categorical structures. On the other hand, the theory $T$ of the Banach lattice 
$$L^1[0,1]=(L^1([0,1],\mathcal{B},\lambda),+,\wedge,\vee,\|\cdot\|)$$ 
is $\aleph_0$-categorical but not exact: take $M\coloneqq L^1[0,1]$, the unique separable model of this theory, and let $f,g$ be any elements of $M$ that are positive and have norm $1$.  Then one has $(M,f) \equiv (M,g)$ by quantifier elimination for $T$. However, if $f$ has full support in $[0,1]$ while $g$ does not, then these two expansions of $M$ are not isomorphic.

The theory \ARV\ of $[0,1]$-valued random variables on an atomless probability space \cite{benOntheories} is also not exact despite being bi-interpretable with the theory of atomless probability algebras---so exactness is not preserved under bi-interpretability.

We recall that $\aleph_0$-categoricity is always preserved by the operation of naming the elements of the algebraic closure of the empty set (see \cite[Prop.~1.15]{benCannot}). So under our assumption it is also preserved by naming the elements of the algebraic closure of a given finite tuple.

We observe as well that this strong form of $\aleph_0$-categoricity is equivalent to saying that $M$ is \emph{homogeneous}, in the sense that for every finite tuples $a,b$ in $M$ with $a\equiv b$ (i.e., having the same type over~$\emptyset$) there is $g\in G$ such that $ga=b$.  Exactness is also equivalent to the statement that the $G$-orbit of every finite tuple is closed, hence definable. In particular, under our hypothesis we may quantify over any $G$-orbit and consider predicates defined only on a given $G$-orbit.

We recall that the \emph{strong type} of $b$ over $a$ is defined by
$$\stp(b/a)\coloneqq\tp(b/\acl(a)),$$
where the algebraic closure $\acl(a)$ is taken in $M^\meq$, the expansion of $M$ obtained by adding all metric imaginary sorts. We write $b\eqstp_a c$ as short for $\stp(b/a)=\stp(c/a)$.

We will need to consider strong types to ensure stationarity in the sense of stability theory, but metric imaginaries and the $M^\meq$ construction will play no other role in the paper. The imaginaries needed are canonical parameters for $M$-definable predicates, which are discussed in \cite[\textsection 5]{benusv10}.  Stationarity of these strong types is shown in \cite[Prop.~8.11(vii)]{benusv10}.
 
The reader may just as well consider only ordinary types and assume that all types over finite sets are stationary (which is the case in our main example, the theory of atomless probability measure algebras).

Given $n,k\in\mbN$ and a $G$-orbit $O\subseteq M^n$, we define $D_{O,k}\colon O\times M^k\times M^k\to\mbR$ by
$$D_{O,k}(a,b,c)\coloneqq d(\stp(b/a),\stp(c/a))$$
for every $a\in O$ and $b,c\in M^k$. Recall that the distance between two types is the infimum of the distance between realizations of the types in elementary extensions. Since here we consider types over the algebraic closure of finite sets, and since $T$ is $\aleph_0$-categorical and exact, this distance is actually achieved by realizations of the types in $M^k$.  (For any finite tuple $a$ from $M$, the expansion of $M$ (or rather, $M^\meq$) by naming the elements of $\acl(a)$ remains $\aleph_0$-categorical in this situation; therefore distances between types are realized in this structure.)  The restriction to a single orbit $O$ will be important, among other things, to ensure the continuity (and hence the definability) of $D_{O,k}$.

\begin{lem}\label{l:DOk-is-def}
For every $n,k\in\mbN$ and every $G$-orbit $O\subseteq M^n$, the predicate $D_{O,k}$ is definable.
\end{lem}
\begin{proof}
The predicate $D_{O,k}$ is $G$-invariant. Indeed, given $a\in O$ and $b,c\in M^k$, take tuples $b',c'\in M^k$ such that $b'\eqstp_a b$, $c'\eqstp_a c$, and $D_{O,k}(a,b,c)=d(b',c')$. If $g\in G$, then $gb'\eqstp_{ga}gb$, 
$gc'\eqstp_{ga}gc$, and
$$D_{O,k}(ga,gb,gc)=D_{O,k}(ga,gb',gc')\leq d(gb',gc')=d(b',c')=D_{O,k}(a,b,c)\text{.}$$
This also applies to $g^{-1}$ in place of $g$, so it implies $G$-invariance.

Now, if we fix an element $a\in O$, the induced map $M^k\times M^k\to\mbR$, $(b,c)\mapsto D_{O,k}(a,b,c)$, is clearly continuous and $\Aut(M/a)$-invariant, hence definable in $(M,a)$ since this expansion is $\aleph_0$-categorical. Thus there is a definable predicate $\varphi$ of $M$ such that for every $b,c\in M^k$ we have $D_{O,k}(a,b,c)=\varphi(a,b,c)$. Since $D_{O,k}$ and $\varphi$ are $G$-invariant and $O=Ga$, it follows that $D_{O,k}(a',b,c)=\varphi(a',b,c)$ for every $a'\in O$ and $b,c\in M^k$. That is, $D_{O,k}$ is definable. 
\end{proof}

From now on we suppose that the theory $T$ is also stable. For the proof of our main theorem in Section~\ref{s:main}, we will need a definability result analogous to the one presented above, for the distance function to non-forking extensions. In the rest of this section, we identify a necessary and sufficient condition for this definability result to hold, namely that the stable independence relation is \emph{continuously symmetric} (see Definition~\ref{d:continuously-symmetric} and Lemma \ref{l:IOkl-is-def}). In Section~\ref{s:particular-cases} we will see that this condition follows from a strong form of homogeneity which is easy to check in concrete examples.

Throughout, the symbol $\ind$ denotes the independence relation given by non-forking. Additionally, given $\epsilon>0$ and tuples $a,b,c$, we will write $b^\epsilon\ind_a c$ to mean that there is $b'$ in an elementary extension such that $b'\eqstp_a b$ and $b'\ind_a c$ and $d(b,b')<\epsilon$. When all tuples are finite, such a $b'$ can be taken in $M$, because the expansion of $M$ (or rather, $M^\meq$) with names for the elements of $\acl(a)c$ is $\aleph_0$-categorical. Similarly, $b\ind_a c^\epsilon$ means that there is $c'$ such that $c'\eqstp_a c$ and $b\ind_a c'$ and $d(c,c')<\epsilon$.

\begin{rem}\label{rem:eps-ind-without-stp}
Given $a,b,c$, if there is any $b'$ such that $d(b,b')<\epsilon$ and $b'\ind_a c$, then $b^{2\epsilon}\ind_a c$. Indeed, in that case let $b''$ be such that $b''\eqstp_{ab'}b$ and $b''\ind_{ab'}c$. Then, by transitivity of independence, $b''\ind_a c$. On the other hand, $d(b'',b)\leq d(b'',b')+d(b',b)<2\epsilon$ because $d(b'',b')=d(b,b')$. Thus $b''$ witnesses the fact that $b^{2\epsilon}\ind_a c$.

Similarly, if there is $c'$ with $d(c,c')<\epsilon$ and $b\ind_a c'$, then $b\ind_a c^{2\epsilon}$.
\end{rem}

\begin{defin}\label{d:continuously-symmetric}
We will say that the independence relation $\ind$ is \emph{continuously symmetric on $M$} if for every $n,k,l\in\mbN$, every $G$-orbit $O\subseteq M^n$ and every $\epsilon>0$, there is $\delta>0$ such that, for every $a\in O$, $b\in M^k$, and $c\in M^l$,
$$\text{if $b^\delta\ind_ac$, then $c^\epsilon\ind_ab$}.$$
\end{defin}

Note that, by symmetry of $\ind$, the implication displayed in the previous definition may be equivalently replaced by any of the following:
\begin{itemize}
\item if $b^\delta\ind_ac$, then $b\ind_ac^\epsilon$;
\item if $b\ind_ac^\delta$, then $b^\epsilon\ind_ac$.
\end{itemize}

Now, given $n,k,l\in\mbN$ and a $G$-orbit $O\subseteq M^n$, we define $I_{O,k,l}\colon O\times M^k\times M^l\to\mbR$ by
$$I_{O,k,l}(a,b,c)\coloneqq \inf\{\epsilon>0: \textstyle{b^\epsilon\ind_a c}\}.$$
Note that $I_{O,k,l}(a,b,c)=d(\stp(b/ac),\stp(b/a){\upharpoonright}^{ac})$, where $\stp(b/a){\upharpoonright}^{ac}$ denotes the unique non-forking extension of $\stp(b/a)$ to a strong type over $ac$. The next result gives the promised characterization of definability for this distance function.

\begin{lem}\label{l:IOkl-is-def}
The following are equivalent:
\begin{enumerate}
\item The stable independence relation is continuously symmetric on $M$.

\item For every $n,k,l\in\mbN$ and every $G$-orbit $O\subseteq M^n$, the predicate $I_{O,k,l}$ is definable.
\end{enumerate}
\end{lem}
\begin{proof}
$(1)\Rightarrow(2):$ The predicate $I_{O,k,l}$ is clearly $G$-invariant. As in the proof of Lemma~\ref{l:DOk-is-def}, it suffices to fix an element $a\in O$ and see that the induced map
$$I_a\colon M^k\times M^l\to\mbR,\  (b,c)\mapsto I_{O,k,l}(a,b,c)$$
is continuous and $\Aut(M/a)$-invariant.

Invariance under $\Aut(M/a)$ follows from $G$-invariance of $I_{O,k,l}$. To see that $I_a$ is continuous, it is enough to see that for every $\epsilon>0$ there is $\eta>0$ such that $|I_a(b,c)-I_a(\tilde b,c)|\leq\epsilon$ whenever $d(b,\tilde b)<\eta$, and $|I_a(b,c)-I_a(b,\tilde c)|\leq\epsilon$ whenever $d(c,\tilde c)<\eta$. Fix $\epsilon>0$ and choose a corresponding $\delta>0$ satisfying the properties of Definition~\ref{d:continuously-symmetric}. Then set $\eta=\min(\delta,\epsilon)/4$.

If $d(b,\tilde b)<\eta$ and $I_a(b,c)<r$, let $b'\eqstp_a b$ be such that $b'\ind_a c$ and $d(b,b')<r$. Now choose $g\in\Aut(M/\acl(a))$ mapping $b$ to $b'$. Then $d(b',g\tilde b)<\eta$, so $(g\tilde b)^{2\eta}\ind_a c$ by Remark~\ref{rem:eps-ind-without-stp}. Since $g\tilde b\eqstp_a \tilde b$ and $d(\tilde b,g\tilde b)\leq d(\tilde b,b)+d(b,b')+d(b',g\tilde b)<r+2\eta$, we obtain that $I_a(\tilde b,c)<r+4\eta$. We conclude that $I_a(\tilde b,c)\leq I_a(b,c)+\epsilon$ and, by symmetry, $|I_a(b,c)-I_a(\tilde b,c)|\leq\epsilon$.

Now suppose $d(c,\tilde c)<\eta$ and $I_a(b,c)<r$. Choose $b'\eqstp_a b$ with $b'\ind_a c$ and $d(b,b')<r$. Since $d(c,\tilde c)<\delta/2$, by Remark~\ref{rem:eps-ind-without-stp} we have $b'\ind_a \tilde c^\delta$ and thus, by the choice of $\delta$ (witnessing that $\ind$ is continuously symmetric), $b'^\epsilon\ind_a \tilde c$. It follows that $I_a(b,\tilde c)<r+\epsilon$. We conclude as before that $|I_a(b,c)-I_a(b,\tilde c)|\leq\epsilon$.

$(2)\Rightarrow(1):$ We fix $O,k,l$ and some $\epsilon>0$. Since $I_{O,k,l}$ is definable, and hence uniformly continuous, there is $\delta>0$ such that $|I_{O,k,l}(a,b,c)-I_{O,k,l}(a,b,c')|<\epsilon$ whenever $d(c,c')<\delta$. Thus, if $b\ind_a c^\delta$, there is $c'$ such that $d(c,c')<\delta$ and $b\ind_a c'$, so $I_{O,k,l}(a,b,c')=0$ and $I_{O,k,l}(a,b,c)<\epsilon$. This means that $b^\epsilon\ind_a c$, proving that independence is continuously symmetric.
\end{proof}

\begin{rem}
Recall that in infinite-dimensional Hilbert spaces, orthogonality (in the sense of the inner product) agrees with forking independence. We will now show that the orthogonality relation in Hilbert spaces is not continuously symmetric. 

Let $u,v,w$ be three unit vectors such that $0<\|u-v\|<\delta/2$ and $u\perp w$, and $w$ is in the plane generated by $u,v$. Then using Remark \ref{rem:eps-ind-without-stp} we have $(uv)^\delta\perp w$, but every unit vector $w'$ with $uv\perp w'$ is at distance $\sqrt{2}$ from $w$. Thus forking independence is not definable in the sense of definability of the predicates $I_{O,k,l}$. This implies that our approach to the axiomatization of existentially closed $\mbF_2$-actions will not apply to Hilbert spaces, although, as recalled in the introduction, we know in this case that they are axiomatizable (\cite{berHilb}).
\end{rem}

We have thus proved that under appropriate assumptions, the predicates $D_{O,k}$ and $I_{O,k,l}$ are definable in the unique separable model $M$ of $T$. Let us remark that if $N$ is an arbitrary model of $T$, then the formulas defining these predicates have the same \emph{intended meaning} in $N$ (i.e., they calculate the distance between strong types and the distance to the non-forking extension, respectively). One should note beforehand that if $O\subseteq M^n$ is a given $G$-orbit, then $O^N$ (i.e., the corresponding definable set as interpreted in $N$) is the set of realizations of the $\emptyset$-type corresponding to $O$. These remarks can be justified by passing to a separable elementary submodel, where by $\aleph_0$-categoricity the formulas considered have the intended meaning, and then noting that this meaning is independent of the ambient model.

\noindent\hrulefill

\section{Existentially closed $\mbF_k$-actions}\label{s:main}

We prove in this section our main result, namely the characterization and axioma-tizability---under appropriate conditions---of the existentially closed models of $T_k$, the theory $T$ expanded with $k$ automorphisms. For notational simplicity we discuss the details in the case $k=2$, but everything holds, \emph{mutatis mutandis}, for any finite $k$, and hence also for $k=\omega$, as we explain in the proof of Theorem~\ref{thm:main} below.\footnote{In fact, the general case (including $k=\omega$) can also be formally deduced from the case $k=2$ by using the trick of \cite[Cor.~2.5]{kikpil}, based on a lemma proved by Chatzidakis.}

In addition to the hypotheses discussed in the previous section, we will require here that the base theory $T$ is \emph{superstable}. We recall that a stable metric theory is superstable if for every $\epsilon>0$ and every finite tuple $b$ and set $C$ in any model, there is a finite subset $A\subseteq C$ such that $b^\epsilon\ind_A C$. Every $\aleph_0$-stable metric theory is superstable; see, for example, \cite[Thm.~4.13]{benCatcat}.

We will also assume $T$ is \emph{model complete}, i.e., every embedding of models of $T$ is elementary, which is equivalent to the fact that every definable predicate $P(x)$ in models of $T$ is \emph{existential} (more appropriately called a \emph{decreasing predicate} or \emph{$\inf$-predicate}), i.e., it can be approximated by formulas of the form $\inf_y\varphi(x,y)$ where $\varphi$ is quantifier-free.

We repeat all our hypotheses for clarity:

\begin{assumption}\label{assumption:main}
Throughout this section, $\mcL$ is a countable metric first-order language and $T$ is a model complete, $\aleph_0$-categorical, exact, superstable $\mcL$-theory with a continuously symmetric stable independence relation.
\end{assumption}

We consider the expansion $\mcL_2=\mcL\cup\{\sigma_1,\sigma_2,\sigma_{-1},\sigma_{-2}\}$ of the language of $T$ with four new unary function symbols, and we define $T_2$ as the $\mcL_2$-theory that extends $T$ and states that $\sigma_1$ and $\sigma_2$ are $\mcL$-automorphisms and that $\sigma_{-1}$ and $\sigma_{-2}$ are their respective inverses. In other words, $(N,\tau_1,\tau_2,\tau_{-1},\tau_{-2})\models T_2$ iff $N\models T$ and $\tau_1,\tau_2,\tau_{-1},\tau_{-2}\in\Aut(N)$ with $\tau_1\inv=\tau_{-1}$ and $\tau_2\inv=\tau_{-2}$. It is easy to see that $T_2$ is a continuous first-order theory. We refer to a model of $T_2$ as an \emph{$\mbF_2$-action} on a model of $T$.

\begin{rem} The inclusion of the inverses in the language (which ensures that the inverses pass to substructures) is not necessary for the results of this section, but it will be important in concrete examples in order to verify the amalgamation hypothesis of Corollary~\ref{cor:qe} below (see Remark~\ref{rem:T2forall=Tforall2} and Theorem~\ref{thm:main-example}). In other words, this seems to be the appropriate language for quantifier elimination. For brevity, however, we shall omit any reference to $\tau_{-1},\tau_{-2}$ and denote the models of $T_2$ simply by $(N,\tau_1,\tau_2)$.
\end{rem}

    The reader may note the following easy fact. For basics on inductive theories in continuous logic, see 
\cite[\textsection 3]{usvyGeneric}. Amalgamation for models of $T_2$ follows from the stability (and model completeness) of $T$, as in \cite[Thm.~3.3]{lasAutour}.

\begin{prop}
The theory $T_2$ is inductive and has the amalgamation property.
\end{prop}

Our main goal is to prove that $T_2$ admits a \emph{model companion}, i.e., a model complete $\mcL_2$-theory $T_2^*$ such that every model of $T_2$ embeds into a model of $T_2^*$ and vice-versa. Since $T_2$ has the amalgamation property, it then follows that $T_2^*$ is in fact a \emph{model completion} of $T_2$, meaning that for every $A\models T_2$, the theory $T_2^*$ together with the diagram of $A$ is a complete theory.

On the other hand, since $T_2$ is inductive, the existence of a model companion for $T_2$ is equivalent to the axiomatizability of the class of its existentially closed models, and in that case the model companion is precisely the common theory of this class (see, for instance, \cite[Thm.~8.3.6]{hodgesMT}, whose proof works as well in the continuous setting). Recall that $A$ is an \emph{existentially closed} model of $T_2$ if for every extension $A\subseteq B$ with $B\models T_2$, every quantifier-free $\mcL_2$-formula $\varphi(x,y)$, every tuples $a\in A^x$ and $b\in B^y$, and every $\epsilon>0$, one can find $c\in A^y$ with $ |\varphi(a,b)-\varphi(a,c)|<\epsilon$.

We will thus define a class $\mcE$ of models of $T_2$, then prove that it is elementary (i.e., axiomatizable) and that it consists precisely of the existentially closed models of $T_2$. We write $b\eqstp_a^\epsilon c$ to abbreviate $d(\stp(b/a),\stp(c/a))<\epsilon$. 

\begin{defin}\label{d:class of ec models}
We define $\mcE$ to be the class of models $(N,\tau_1,\tau_2)\models T_2$ such that for all $\epsilon>0$ and for every set of finite tuples $a,b_0,b_1,b_2$ in $N$ satisfying the conditions:
\begin{equation}\label{axiom:hyp}\tag{C1}
b_1 \eqstp_{\tau_1(a)}^\epsilon \tau_1(b_0),\ \ b_2\eqstp_{\tau_2(a)}^\epsilon \tau_2(b_0),\ 
\ b_0^\epsilon\ind_a \tau_1(a)\tau_2(a),
\ \ b_1^\epsilon \ind_{\tau_1(a)} a\tau_2(a),\ \ b_2^\epsilon\ind_{\tau_2(a)} a\tau_1(a),
\end{equation}
there is a tuple $c$ in $N$ such that
\begin{equation}\label{axiom:con}\tag{C2}
b_0b_1b_2 \eqstp_a^{2\epsilon} c\tau_1(c)\tau_2(c).
\end{equation}
\end{defin}

Informally, the class $\mcE$ consists of those $\mbF_2$-actions on models of $T$ with the following genericity property: whenever some tuples $b_1$ and $b_2$ behave as the translates of a tuple $b_0$ by $\tau_1$ and $\tau_2$ respectively---relative to some given parameter $a$ and satisfying appropriate independence conditions with respect to $\tau_1(a)$ and $\tau_2(a)$---, there is a tuple $c$ in the model such that $c,\tau_1(c),\tau_2(c)$ behave as $b_0,b_1,b_2$ over $a$. The independence conditions ensure that such a behaviour can be realized in an extension of the action; and this explains why existentially closed actions belong to $\mcE$. Conversely, if some finite configuration with parameters is realized by $b,\tau_1(b),\tau_2(b)$ in an extension of the action, then the model completeness and superstability of $T$ permit to find appropriate tuples $b_0,b_1,b_2$ in a similar position in the model, together with a possibly larger tuple of parameters $a$ satisfying the hypotheses (\ref{axiom:hyp}); if the action belongs to $\mcE$, the condition (\ref{axiom:con}) then allows to conclude that it is existentially closed.

\begin{prop}\label{p:E-is-elementary}
The class $\mcE$ is elementary.
\end{prop}
\begin{proof}
For every $n,k\in\mbN$ and every $G$-orbit $O\subseteq M^n$, let $A_{O,k}$ be the condition:
$$\forall a\in O\ \forall b_0,b_1,b_2\in M^k\left( \inf_{c\in M^k}\Xi_0 \leq 2\max(\Xi_1,\Xi_2,\Psi_0,\Psi_1,\Psi_2)\right),$$
where:
\begin{itemize}
\item $\Xi_0=D_{O,3k}(a,b_0b_1b_2,c\tau_1(c)\tau_2(c))$,
\item $\Xi_1=D_{O,k}(\tau_1(a),b_1,\tau_1(b_0))$,
\item $\Xi_2=D_{O,k}(\tau_2(a),b_2,\tau_2(b_0))$,
\item $\Psi_0=I_{O,k,2n}(a,b_0,\tau_1(a)\tau_2(a))$,
\item $\Psi_1=I_{O,k,2n}(\tau_1(a),b_1,a\tau_2(a))$,
\item $\Psi_2=I_{O,k,2n}(	\tau_2(a),b_2,a\tau_1(a))$.
\end{itemize}
Here, $D_{O,k}$ and $I_{O,k,2n}$ are the predicates introduced in the previous section, which are definable by Lemma~\ref{l:DOk-is-def} and Lemma~\ref{l:IOkl-is-def}. To see that $A_{O,k}$ is first-order expressible in continuous logic, observe that any condition of the form $\forall x(\varphi_0\leq\varphi_1)$ can be rewritten as $\sup_x\max(\varphi_0-\varphi_1,0)=0$. We recall as well that we may quantify over each definable set, such as each orbit~$O$.

Thus we see each $A_{O,k}$ as a continuous first-order sentence in the language $\mcL_2$. Next we check that the theory $T_2\cup \{A_{O,k}\}_{O,k}$ axiomatizes the class $\mcE$.

Suppose $(N,\tau_1,\tau_2)\in\mcE$ and that we are given $O,k$ and tuples $a\in O^N$ and $b_0,b_1,b_2\in N^k$. If $\epsilon>0$ is any positive real such that the tuples satisfy that $\max(\Xi_1,\Xi_2,\Psi_0,\Psi_1,\Psi_2)<\epsilon$, then the hypotheses (\ref{axiom:hyp}) in the definition of $\mcE$ are satisfied. Hence  for some $c\in N^k$ we have the conclusion (\ref{axiom:con}), which implies that $\inf_{c\in M^k}\Xi_0<2\epsilon$. We deduce that $\inf_{c\in N^k}\Xi_0 \leq 2\max(\Xi_1,\Xi_2,\Psi_0,\Psi_1,\Psi_2)$. Hence, $(N,\tau_1,\tau_2)\models A_{O,k}$.

Conversely, suppose $(N,\tau_1,\tau_2)\models A_{O,k}$ and that $a\in O^N$ and $b_0,b_1,b_2\in N^k$ are tuples satisfying the conditions (\ref{axiom:hyp}) for a given $\epsilon>0$. It follows that the tuples satisfy $\max(\Xi_1,\Xi_2,\Psi_0,\Psi_1,\Psi_2)<\epsilon$, and thus also $\inf_{c\in N^k}\Xi_0<2\epsilon$. This says precisely that there is $c\in N^k$ such that the conclusion (\ref{axiom:con}) holds, as desired.
\end{proof}
\begin{prop}\label{p:e.c.-models-are-in-E}
Every existentially closed model of $T_2$ belongs to $\mcE$.
\end{prop}
\begin{proof}
Let $(M,\tau_1,\tau_2)\models T_2$ be existentially closed and let $a,b_0,b_1,b_2$ be tuples in $M$ satisfying the conditions (\ref{axiom:hyp}). The independence hypotheses say that there are tuples $b_0',b_1',b_2'$ in $M$ such that $d(b_0b_1b_2,b_0'b_1'b_2')<\epsilon$ and
\begin{equation*}
b_0'\eqstp_a b_0,\ \ b_0'\ind_a \tau_1(a)\tau_2(a),
\ \
b_1'\eqstp_{\tau_1(a)} b_1,\ \ b_1'\ind_{\tau_1(a)}a\tau_2(a),
\ \
b_2'\eqstp_{\tau_2(a)} b_2,\ \ b_2'\ind_{\tau_2(a)}a\tau_1(a).
\end{equation*}
Then $b_1'\eqstp_{\tau_1(a)}^\epsilon \tau_1(b_0')$, so there is an auxiliary $b_1''$ such that $d(b_1',b_1'')<\epsilon$ and $b_1''\eqstp_{\tau_1(a)} \tau_1(b_0')$. Take furthermore $b_1'''$ a realization of $\stp(b_1''/\tau_1(a)b_1'){\upharpoonright}^{\tau_1(a)b_1' a\tau_2(a)}$, i.e., such that $b_1'''\eqstp_{\tau_1(a)b_1'}b_1''$ and $b_1'''\ind_{\tau_1(a)b_1'}a\tau_2(a)$. Then $b_1'''\ind_{\tau_1(a)}a\tau_2(a)$ by transitivity of independence, and we also have $d(b_1''',b_1)\leq d(b_1''',b_1')+d(b_1',b_1)=d(b_1'',b_1')+d(b_1',b_1)<2\epsilon$. In other words, we may choose $b_1'$ appropriately (replacing it by $b_1'''$) so as to have
$$b_1'\eqstp_{\tau_1(a)} \tau_1(b_0'),\ \ b_1'\ind_{\tau_1(a)}a\tau_2(a),\ \ d(b_1',b_1)<2\epsilon.$$
Similarly, we can choose $b_2'$ so that it satisfies
$$b_2'\eqstp_{\tau_2(a)} \tau_2(b_0'),\ \ b_2'\ind_{\tau_2(a)}a\tau_1(a),\ \ d(b_2',b_2)<2\epsilon.$$

Now let $M'\succeq M$ be a sufficiently saturated and homogeneous model of $T$, and choose tuples $c_0,c_1,c_2$ in $M'$ such that
$$c_0c_1c_2 \eqstp_{a\tau_1(a)\tau_2(a)} b_0'b_1'b_2' \mbox{\ \ and\ \ } c_0c_1c_2\ind_{a\tau_1(a)\tau_2(a)} M.$$
By transitivity,
\begin{equation}\label{eq:c0-idep-a-M-etc}
c_0\ind_a M,\ \ c_1\ind_{\tau_1(a)}M, \mbox{\ \ and\ \ }c_2\ind_{\tau_2(a)}M.
\end{equation}
On the other hand, if $a'$ is an enumeration of $\acl(a)$ and $m$ is an enumeration of $M$, then $c_0a'\equiv b_0'a'\equiv \tau_1(b_0')\tau_1(a')\equiv b_1'\tau_1(a') \equiv c_1\tau_1(a')$, and similarly for $c_2$ and $\tau_2$. Thus
$$c_0a'\equiv c_1\tau_1(a')\equiv c_2\tau_2(a'),\ \ a'm\equiv \tau_1(a')\tau_1(m) \equiv \tau_2(a')\tau_2(m),$$
so by stationarity of independence and (\ref{eq:c0-idep-a-M-etc}), we obtain that
$$c_0m\equiv c_1\tau_1(m)\equiv c_2\tau_2(m).$$
Hence we can extend $\tau_1,\tau_2$ to automorphisms $\tau_1',\tau_2'$ of $M'$ in such a way that
$$\tau_1'(c_0)=c_1 \mbox{\ \ and\ \ } \tau_2'(c_0)=c_2.$$
Now, since $c_0c_1c_2\eqstp_a b_0'b_1'b_2'$ and $d(b_0'b_1'b_2',b_0b_1b_2)<2\epsilon$, we have that
$$b_0b_1b_2\eqstp_a^{2\epsilon} c_0\tau_1'(c_0)\tau_2'(c_0).$$
To summarize, we have an extension $(M',\tau_1',\tau_2')\supseteq (M,\tau_1,\tau_2)$ that is a model of $T_2$ and has a witness $c=c_0$ for the conclusion (\ref{axiom:con}) in the property defining $\mcE$. In other words, for the appropriate $O$ and $k$,
\begin{equation}\label{eq:M'12-models-DOk}
(M',\tau_1',\tau_2')\models \inf_c D_{O,k}(a,b_0b_1b_2,c\sigma_1(c)\sigma_2(c)) < 2\epsilon.
\end{equation}
Since $T$ is model complete, the $\mcL$-predicate $D_{O,k}$ is existential. Thus, the $\mcL_2$-condition in (\ref{eq:M'12-models-DOk}) is existential. As $(M,\tau_1,\tau_2)$ is existentially closed, we can find the desired witness $c$ already in~$M$.
\end{proof}

Let us write $b\equiv^\epsilon c$ as short for $d(\tp(b),\tp(c))<\epsilon$. The following equivalence is true for any model complete, $\aleph_0$-categorical theory $T$.

\begin{lem}\label{l:exist-closedness}
Let $(M,\tau_1,\tau_2)\subseteq (M',\tau_1',\tau_2')$ be an extension of models of $T_2$. Then $(M,\tau_1,\tau_2)$ is existentially closed in $(M',\tau_1',\tau_2')$ if and only if for every $\epsilon>0$ and every finite tuples $b$ in $M'$ and $a$ in $M$ there is a tuple $c$ in $M$ such that
$$ab\tau_1'(b)\tau_2'(b)\equiv^\epsilon ac\tau_1(c)\tau_2(c).$$
\end{lem}
\begin{proof}
The left-to-right implication follows from existential definability of the distance function to $\tp(ab\tau'_1(b)\tau'_2(b))$. The converse can be seen easily by rewriting any existential $\mcL_2$-condition as an equivalent existential $\mcL_2$-condition using a larger set of variables and involving only the function symbols $\sigma_1,\sigma_2$ and no concatenation thereof. For a formal argument, let $\Sigma$ be the monoid of words in the letters $\{\sigma_1,\sigma_2,\sigma_{-1},\sigma_{-2}\}$, and let us consider its action $\Sigma\actson M'$ where $\sigma_1$ acts as $\tau_1'$, $\sigma_2$ acts as $\tau_2'$, $\sigma_{-1}$ acts as ${\tau_1'}\inv$, $\sigma_{-2}$ acts as ${\tau_2'}\inv$, and the empty word acts as the identity. Given a finite subset $\Delta\subseteq\Sigma$ and a tuple $b$ from $M'$, we let $\Delta b$ denote the tuple $(wb)_{w\in\Delta}$. To show that $(M,\tau_1,\tau_2)$ is existentially closed in $(M',\tau_1',\tau_2')$ it is enough to show that for every $\epsilon>0$, every finite subset $\Delta\subseteq\Sigma$ and every finite tuples $b$ in $M'$ and $a$ in $M$ there is a tuple $c$ in $M$ such that $a\Delta b\equiv^\epsilon a\Delta c$. Moreover, if we let $\Delta_n\subseteq\Sigma$ be the set of all words of length less than or equal to $n\in\mbN$, we may restrict our attention to these finite sets.

Let $n\geq 1$ be arbitrary, and let $b$ and $a$ be finite tuples from $M'$ and $M$, respectively. Let also $\tilde b=\Delta_n b$. Given $\epsilon>0$, we know by hypothesis that there is $\tilde c=(c_w)_{w\in\Delta_n}$ in $M$ such that
\begin{equation}\label{eq:bbb-eqstp-ccc}
a\tilde b\tau_1'(\tilde b)\tau_2'(\tilde b) \equiv^{\epsilon/4n} a\tilde c\tau_1(\tilde c)\tau_2(\tilde c).
\end{equation}
Take a non-empty word $w\in\Delta_n$. Suppose first that $w=\sigma_i w'$ with $i\in\{1,2\}$. Then we have, trivially,
\begin{equation*}
wb=\tau_i'(w'b),\ \ wb\subseteq \tilde b,\ \ \tau_i'(w'b)\subseteq\tau_i'(\tilde b),
\end{equation*}
and these conditions together with (\ref{eq:bbb-eqstp-ccc}) imply readily that $d(c_w,\tau_i(c_{w'}))<\epsilon/2n$. If we suppose, instead, that $w=\sigma_{-i} w'$ with $i\in\{1,2\}$, then
\begin{equation*}
w'b=\tau_i'(wb),\ \ w'b\subseteq \tilde b,\ \ \tau_i'(wb)\subseteq\tau_i'(\tilde b),
\end{equation*}
and these together with (\ref{eq:bbb-eqstp-ccc}) imply that $d(c_w,\tau_i\inv(c_{w'}))=d(c_{w'},\tau_i(c_w))<\epsilon/2n$.

Letting $c=c_\emptyset$, we deduce that
$$d(c_w,wc)<\epsilon/2$$
for each $w\in\Delta_n$. We conclude that $a\tilde b\equiv^{\epsilon/2} a (c_w)_{w\in\Delta_n}\equiv^{\epsilon/2} a\Delta_n c$, and hence that $a\Delta_n b\equiv^\epsilon a\Delta_n c$, as desired.
\end{proof}

\begin{prop}\label{p:models-in-E-are-e.c.}
Every model in $\mcE$ is an existentially closed model of $T_2$.
\end{prop}
\begin{proof}
Let $(M,\tau_1,\tau_2)\in\mcE$ and let $(M,\tau_1,\tau_2)\subseteq (M',\tau_1',\tau_2')$ be an extension of models of $T_2$. In particular $M\preceq M'$ as $\mcL$-structures, because $T$ is model complete. Using the previous lemma, it is enough to show that for every $\epsilon>0$ and every finite tuples $b$ in $M'$ and $a$ in $M$ there is a tuple $c$ in $M$ such that $b\tau_1'(b)\tau_2'(b)\eqstp_a^\epsilon c\tau_1(c)\tau_2(c)$.

Let $\Gamma$ be the group generated by $\tau_1',\tau_2'$ inside $\Aut(M')$. By superstability, there is a finite subset $\Delta\subseteq\Gamma$ such that $b^{\epsilon/2} \ind_{\Delta a}\Gamma a$, and we may assume that $\Delta$ contains the identity element. Let $\tilde a=\Delta a$. Then we have
$$b^{\epsilon/2}\ind_{\tilde a}\tau_1(\tilde a)\tau_2(\tilde a),\ \ \tau_1'(b)^{\epsilon/2} \ind_{\tau_1(\tilde a)}\tilde a\tau_2(\tilde a),\ \ \tau_2'(b)^{\epsilon/2}\ind_{\tau_2(\tilde a)}\tilde a\tau_1(\tilde a).$$
Since $M\preceq M'$ and the expansion $(M,\tilde a\tau_1(\tilde a)\tau_2(\tilde a))$ is $\aleph_0$-categorical, there are $b_0,b_1,b_2$ in $M$ such that
$$b_0b_1b_2\eqstp_{\tilde a\tau_1(\tilde a)\tau_2(\tilde a)}b\tau_1'(b)\tau_2'(b).$$
Hence we have
$$b_1 \eqstp_{\tau_1(\tilde a)} \tau_1(b_0),\ \ b_2\eqstp_{\tau_2(\tilde a)} \tau_2(b_0),\ 
\ b_0^{\epsilon/2}\ind_{\tilde a} \tau_1(\tilde a)\tau_2(\tilde a),
\ \ b_1^{\epsilon/2} \ind_{\tau_1(\tilde a)} \tilde a\tau_2(\tilde a),\ \ b_2^{\epsilon/2}\ind_{\tau_2(\tilde a)} \tilde a\tau_1(\tilde a),$$
which falls within the hypotheses (\ref{axiom:hyp}) of the property defining $\mcE$. Thus we can find $c$ in $M$ such that
$$c\tau_1(c)\tau_2(c)\eqstp_{\tilde a}^\epsilon b_0b_1b_2 \eqstp_{\tilde a}^{} b\tau_1'(b)\tau_2'(b).$$
In particular, $c\tau_1(c)\tau_2(c)\eqstp_a^\epsilon b\tau_1'(b)\tau_2'(b)$, as we wanted.
\end{proof}

\begin{rem}
The hypothesis that the stable independence relation is continuously symmetric was used only for the axiomatization in Proposition~\ref{p:E-is-elementary}, but not for proving that $\mcE$ consists precisely of the existentially closed models of $T_2$.
\end{rem}

We have thus proved that the theory $T_2$ has a model completion. More generally, if we consider $T_k$ (defined in the obvious way) for finite $k$, we may modify the conditions (\ref{axiom:hyp}) in Definition~\ref{d:class of ec models} to have
$$b_i \eqstp_{\tau_i(a)}^\epsilon \tau_i(b_0)\text{  and }\ b_i^\epsilon\ind_{\tau_i(a)} \{\tau_j(a):j\in\{0,\dots,k\}\setminus\{i\}\}$$
for every $i=0,\dots,k$, with the convention that $\tau_0$ is the identity; then change the condition (\ref{axiom:con}) into
$$b_0b_1\dots b_k \eqstp_a^{2\epsilon} c\tau_1(c)\dots\tau_k(c).$$
With the obvious adaptations, the same arguments as above show that the corresponding class $\mcE_k$ of $\mcL_k$-structures is axiomatizable and consists precisely of the existentially closed models of $T_k$.

We can thus conclude the following, always under Assumption~\ref{assumption:main}:

\begin{theorem}\label{thm:main} For every $k\leq\omega$, the theory $T_k$ of $\mbF_k$-actions on models of $T$ has a model completion.
\end{theorem}
\begin{proof}
For the case $k=\omega$, it suffices to observe that an $\mbF_\omega$-action on a model $M \models T$ is e.c.\ if and only if its restriction to the first $k$ generators gives an e.c.\ $\mbF_k$-action (on $M$) for every finite $k$.  The right-to-left direction is immediate, since each formula involves only finitely many symbols. For the left-to-right direction suppose we have an e.c.\ action of $\mbF_\omega$ on a model $M \models T$, with the generators of $\mbF_\omega$ corresponding to automorphisms $(\tau_i \mid i \geq 1)$ of $M$.  Fix $k \geq 1$.  We want to show that the action of $\mbF_k$ on  $M$ determined by $(\tau_1,\dots,\tau_k)$ is e.c..  Let $M' \models T$ contain $M$ and suppose $(\tau'_1,\dots,\tau'_k)$ are automorphisms of $M'$ that extend $(\tau_1,\dots,\tau_k)$. Since $M\preceq M'$ as $\mcL$-structures, we may extend $M'$ to a highly homogeneous $N \models T$ such that any automorphism of $M'$ or of $M$ can be extended to an automorphism of $N$. Therefore we may obtain automorphisms $(\tilde\tau_i \mid i \geq 1)$ of $N$ such that $\tilde\tau_i$ extends $\tau'_i$ for $1 \leq i \leq k$ and $\tilde\tau_i$ extends $\tau_i$ for $k < i$.  This gives us an $\mbF_\omega$-action on $N$ that extends the given one on $M$. It follows that any existential condition that holds in $(M',\tau'_1,\dots,\tau'_k)$ must hold exactly in $(N,\tilde\tau_1,\dots,\tilde\tau_k)$ and thus holds approximately in $(M,\tau_1,\dots,\tau_k)$.
\end{proof}

As is the case for the model companion of any companionable theory, the model completion $T_k^*$ of $T_k$ is complete if and only if any two models of $T_k$ can be jointly embedded into a third one, and $T_k^*$ has quantifier elimination if and only if the universal part of $T_k$---denoted by $(T_k)_\forall$, and whose models are precisely the substructures of models of $T_k$---has the amalgamation property. In our case, quantifier elimination implies, in turn, that $T_k^*$ is stable. Moreover, the stable independence relation in models of $T_k^*$, as well as the algebraic closure operator, can be easily described in terms of the corresponding relation and operator in models of $T$. We note these facts:

\begin{cor}\label{cor:qe} Let $T_k^*$ be the model completion of $T_k$, for any $k\leq\omega$.
\begin{itemize}[wide]
\item If $T_k$ has the joint embedding property, then $T_k^*$ is complete.

\item If $(T_k)_\forall$ has the amalgamation property, then $T_k^*$ has quantifier elimination and is stable. Moreover, in this case, for any model $(M,\bar\tau)\models T_k^*$ and subsets $A,B,C\subseteq M$, if we denote by $\langle S\rangle$ the $\mcL_k$-substructure generated by $S\subseteq M$, we have:
\begin{itemize}[wide=20pt]
\item $A$ is independent from $B$ over $C$ in the sense of $(M,\bar\tau)$ if and only if $\langle A\rangle\ind_{\langle C\rangle}\langle B\rangle$, i.e., iff the generated $\mbF_k$-actions are independent in the sense of the reduct $M$.
\item The (real) algebraic closure of $A$ in $(M,\bar\tau)$ is precisely $\acl(\langle A\rangle)\cap M$.
\end{itemize}
\end{itemize}
\end{cor}
\begin{proof}
Everything follows from very standard arguments, but we give some details for the convenience of the reader. The first point is immediate from the fact that every model of $T_k$ embeds into a model of $T_k^*$, and the model completeness of $T_k^*$. For the second, the amalgamation property for substructures of models of $T_k$, the fact that every model of $T_k$ embeds into a model of $T_k^*$, and the model completeness of $T_k^*$, readily imply the condition of the criterion for quantifier elimination in \cite[Prop.~13.6]{bbhu08}. On the other hand, the stability of $T$ implies that all quantifier-free $\mcL_k$-formulas are stable, so if $T_k^*$ has quantifier elimination then it is also stable.

For the moreover part, let $\ind^*$ denote the ternary relation defined by $A\ind_C^* B$ iff $\langle A\rangle\ind_{\langle C\rangle}\langle B\rangle$. Note that since for $S,S'\subseteq M$, $\langle S\cup S' \rangle$ coincides with the $\mcL$-substructure generated by $\langle S\rangle\cup \langle S'\rangle$, we also have $A\ind_C^* B$ iff $\langle AC \rangle\ind_{\langle C\rangle}\langle BC\rangle$. It suffices as per \cite[Thm.~14.14]{bbhu08} to show that $\ind^*$ satisfies the properties listed in \cite[Thm.~14.14]{bbhu08}. Most properties follow readily from the corresponding properties of $\ind$, except for \emph{stationarity}, \emph{local character} and \emph{extension}. Stationarity of $\ind^*$ follows easily from stationarity of $\ind$ plus quantifier elimination of $T_k^*$. For local character: let $a$ be a finite tuple and $B$ a set, and let $\Delta_n\subseteq \mbF_k$ be an increasing sequence of finite subsets whose union is $\mbF_k$. For each $n\geq 1$ there is a countable set $C_n\subseteq \langle B\rangle$ such that $\Delta_n a\ind_{C_n}\langle B\rangle$. Then $C=\bigcup_n C_n\subseteq \langle B\rangle$ is countable, and it is easy to see that $a\ind_C^*B$.

The extension property of $\ind^*$ is essentially the so-called \emph{PAPA} in \cite{lasBeaux} and \cite{chapil}, but for tuples of automorphisms. Specifically, given $A,B,C\subseteq M$, there is (by the extension property of $\ind$ within $M^\meq$) a set $S$ in an elementary extension of $M$ such that $S\eqstp_{\langle C\rangle}\langle A\rangle$ and $S\ind_{\langle C\rangle}\langle B\rangle$. Let $\theta$ be an $\mcL$-automorphism of a further extension of $M$ (to which, we may assume, the automorphisms $\tau_i$ extend) sending $\langle A\rangle$ to $S$ and fixing $\acl(\langle C\rangle)$ pointwise in $M^\meq$. If for each $i$ we let $\sigma_i=\theta\tau'_i\theta\inv$ for some extension $\tau'_i$ of $\tau_i$, then $\sigma_i|_{\acl(\langle C\rangle)}=\tau_i|_{\acl(\langle C\rangle)}$. Now this and the independence condition $S\ind_{\langle C\rangle}\langle B\rangle$ imply (by stationarity of $\ind$, as in the proof of \cite[Thm.~3.3]{lasAutour}) that the maps $\sigma_i|_S\cup\tau_i|_{\langle BC\rangle}$ are elementary. Thus, by quantifier elimination of $T_k^*$, the $\mbF_k$-action induced on $S\cup\langle BC\rangle$ by these maps can be $\mcL_k$-embedded over $\langle BC\rangle$ into an elementary extension of $(M,\bar\tau)$. Again by quantifier elimination, and since $\theta$ fixes $\langle C\rangle$, the image $A'$ of $\theta(A)$ by this embedding has the same $\mcL_k$-type over $C$ as $A$, and satisfies $A'\ind^*_CB$, as desired.

Concerning the algebraic closure, if $b\in M$ is algebraic over $A$ in the structure $(M,\bar\tau)$, then $b$ is independent from itself over $A$, and from our characterization of independence we obtain that $b\ind_{\langle A\rangle}b$. But then $b$ is algebraic over $\langle A\rangle$ in the sense of $M$.
\end{proof}

We end this section with some remarks and comments.

\begin{rem}\label{rem:T2forall=Tforall2}
The theory $(T_k)_\forall$ always implies the theory $(T_\forall)_k$ (because the inverses of the automorphisms are in the language), and if $T$ has quantifier elimination then $(T_k)_\forall\equiv (T_\forall)_k$. In particular, if $T$ has quantifier elimination and $T_k^*$ exists, then $T_k^*$ is also the model completion of $(T_\forall)_k$.
\end{rem}

\begin{rem}
If $T$ and $T'$ are bi-interpretable model complete theories and $T_k$ has a model companion, then so does $T'_k$. This can be used to conclude the existence of a model completion of $T'_k$ for some theories $T'$ that do not fall within the scope of Assumption~\ref{assumption:main}. Thus, for instance, our results for the theory of $\mbF_k$-actions on probability algebras (see Subsection~\ref{ss:meas-alg}) can be transferred to the theory $\ARV_k$ of $\mbF_k$-actions on spaces of $[0,1]$-valued random variables.
\end{rem}

\begin{rem}
In the discrete first-order setting, if $T$ is superstable and has a model companion $T_1^*$, then $T_1^*$ need not be stable, but, if it is, then it is superstable (and if it is not stable, then it is still \emph{supersimple} \cite[Cor.~3.8]{chapil}). This property may fail in the continuous setting. For example, when $T$ is the theory of atomless probability algebras, then $T_1^*$ exists and is stable, but not superstable \cite[Thm. 3.6]{BB-perturbations}. On the other hand, in that particular case, $T_1^*$ is $\aleph_0$-stable \emph{up to perturbations of the automorphism} \cite[Thm. 2.12]{BB-perturbations}.
\end{rem}

\begin{rem}
Let $T$ be a model complete theory in classical first-order logic. By results of Kikyo \cite{Ki}, if $T$ is unstable and NIP, or unstable and such that $T_1$ has the amalgamation property, then $T_1$ does not have a model companion. This covers for instance the cases of dense linear orders and of the random graph. (It is an open question whether Kikyo's arguments can be made to work in the continuous setting.)
\end{rem}

\noindent\hrulefill
\section{Metric generics and existential closedness}\label{s:metric-generics}

Throughout this section, $\mcL$ is a countable metric first-order language, $M$ is a separable $\mcL$-structure, and $G$ is the automorphism group of $M$ endowed with the topology of pointwise convergence.  

Here we provide a proof that if the theory of $M$ is $\aleph_0$-categorical and model complete, and an $\mbF_k$-action on $M$ is generic in the topological (or rather, \emph{topometric}) sense, then it is existentially closed (see Theorem \ref{thm:metric-gen-is-exist-closed}). This fact was known to several people, but not yet published. We begin with some background.

A system of basic neighborhoods of the identity of $G$ is given by
\begin{equation}\label{eq:basicneighb}
G_{a,\epsilon}=\{g\in G:d(ga,a)<\epsilon\}
\end{equation}
for $n\geq 1$, $a\in M^n$, and $\epsilon>0$. This makes $G$ into a Polish group. In addition, we equip $G$ with a metric $\partial$, the \emph{uniform metric}, given by
$$\partial(g,h)=\sup_{a\in M}d(ga,ha).$$
The uniform metric is complete and refines the topology of $G$, but in general is not Polish. Note that the metric $\partial$ is also left and right invariant. We follow the usual convention that, unless otherwise stated, all purely topological notions (e.g., when we say that a subset is dense, or comeager) refer to the Polish topology of $G$, and not the one induced by the uniform metric.

We extend $\partial$ to the finite powers of $G$ by $\partial(\bar{g},\bar{h})=\max_i\partial(g_i,h_i)$. Given $k<\omega$ and $A\subseteq G^k$, we denote the \emph{uniform closure} of $A$ by $\ov{A}^\partial$. If we are also given $\eta>0$, we define the $\eta$-fattening of $A$ as
\begin{equation}\label{eq:unifneighb}
(A)_\eta=\{\bar{h}\in G^k:\partial(\bar{h},A)<\eta\}.
\end{equation}
Thus $\ov{A}^\partial = \bigcap_{\eta>0}(A)_\eta$.

We will also consider two group actions on $G^k$. On the one hand we have the diagonal action of $G$ by conjugation: given 
$\bar{g}\in G^k$, $h\in G$, and $S\subseteq G$, we denote
$$h\cdot\bar{g}=(hg_1h\inv,\dots,hg_kh\inv) \mbox{\ \ and\ \ } S\cdot\bar{g}=\{h\cdot\bar{g}:h\in S\}.$$
On the other hand, we will consider the canonical group law in the powers $G^k$. That is, if $\bar{g},\bar{h}\in G^k$ and $T\subseteq G^k$, then
$$\bar{h}\bar{g}=(h_1g_1,\dots,h_kg_k) \mbox{\ \ and\ \ } \bar{h}T=\{\bar{h}\bar{g}:\bar{g}\in T\}.$$

The following notion comes from \cite{BBM-topometric}.

\begin{defin}
The group $G$ has \emph{metric $k$-generics} if there is a diagonal conjugacy class in $G^k$ whose uniform closure is comeager in the product topology. In symbols: there is $\bar{g}\in G^k$ such that $\ov{(G\cdot\bar{g})}^\partial$ is comeager in $G^k$. In this case we say that $\bar{g}$ is a \emph{metric generic} of $G^k$, and also that the induced action of $\mbF_k$ on $M$ is \emph{metrically generic}.
\end{defin}

\begin{lem}\label{l:baire}
Let $\bar{g}\in G^k$ be a metric generic. Then for every neighborhood $U$ of the identity of $G$ and every $\eta>0$ there is a neighborhood $V$ of $\bar{g}$ such that $(U\cdot\bar{g})_\eta$ is dense in $V$.
\end{lem}
\noindent Note that the conclusion may be restated as follows: for every $u\in M^n$ and $\epsilon,\eta>0$ there are $v\in M^m$ and $\delta>0$ such that $(G_{u,\epsilon}\cdot\bar g)_\eta$ is dense in $\bar g G_{v,\delta}^k$ (where $G_{v,\delta}^k\coloneqq (G_{v,\delta})^k\subseteq G^k$).

\begin{proof}
Fix a neighborhood $U$ of the identity of $G$, and $\eta>0$. By the continuity of the group operation and the inverse function we can find $W\subseteq G$ an open set such that $W\inv W\subseteq U$. Let $(g_n)_n\subseteq G$ be a sequence such that $G=\bigcup_n g_nW$. The uniform closure of $G\cdot\bar g$ is contained in
$$(G\cdot\bar g)_{\eta/2}=\bigcup_n (g_nW\cdot\bar g)_{\eta/2}.$$
Hence, some set $(g_n W\cdot\bar g)_{\eta/2}$ is non-meager. This implies that $(W\cdot\bar g)_{\eta/2}$ is non-meager, and in particular somewhere dense.

Let $V'\subseteq G^k$ be a non-empty open set such that $(W\cdot\bar g)_{\eta/2}$ is dense in $V'$, and take $\bar g'\in V'$ and $h\in W$ such that $\partial(h\cdot\bar g,\bar g')<\eta/2$. Thus $\bar g=\bar u(h\inv\cdot\bar g')$ for some $\bar u\in G^k$ with $\partial(\bar u,\bar 1)<\eta/2$. Since $(h\inv W\cdot\bar g)_{\eta/2}$ is dense in $h\inv\cdot V'$ and $h\inv W\subseteq U$, we have that $(U\cdot\bar g)_{\eta/2}$ is dense in $h\inv\cdot V'$. Also, since $\bar{u}(U\cdot \bar g)_{\eta/2}\subseteq (U\cdot\bar g)_\eta$, we have that $(U\cdot\bar g)_\eta$ is dense in $V=\bar{u}(h\inv\cdot V')$. The last set is an open set containing~$\bar g$.
\end{proof}

Let $T$ be the $\mcL$-theory of the structure $M$. We let $T_k$ be the theory $T$ expanded with $k$ automorphisms, as in the previous section. Given $\bar g\in G^k$, we denote by $(M,\bar g)$ the induced model of $T_k$.

\begin{theorem}\label{thm:metric-gen-is-exist-closed}
Suppose $T$ is model complete and $\aleph_0$-categorical, and that there is a metric generic $\bar g\in G^k$. Then $(M,\bar g)$ is an existentially closed model of $T_k$.
\end{theorem}
\begin{proof}
Suppose $(M,\bar g)\subseteq (N,\bar\tau)$ where $N$ is a model of $T$ and $\bar\tau\in\Aut(N)^k$. Let $a$ and $b$ be finite tuples from $M$ and $N$, respectively, and take $\epsilon>0$. By Lemma~\ref{l:exist-closedness} (adapted to $T_k$ in the obvious way), it is enough to show that there is a tuple $c$ in $M$ with
$$ab\bar{\tau}(b)\equiv^\epsilon ac\bar{g}(c),$$
where $\bar{\tau}(b)$ denotes the tuple $(\tau_1b,\dots,\tau_kb)$ and $\bar{g}(c)$ denotes $(g_1c,\dots,g_k c)$. By  Lemma~\ref{l:baire}, we can choose a finite tuple $v$ from $M$ and $\delta>0$ such that $(G_{a,\epsilon/4}\cdot \bar g)_{\epsilon/4}$ is dense in $\bar g G_{v,\delta}^k$.

Now, by model completeness and $\aleph_0$-categoricity of $T$, there are tuples $c_0,\dots,c_k\in M^{|b|}$ such that, for $\eta=\min(\epsilon/6,\delta/2)$, we have
$$av\bar{g}(v)b\bar{\tau}(b)\equiv^\eta av\bar{g}(v)c_0c_1,\dots,c_k.$$
In particular, $vc_0\equiv^{2\eta} g_i(v)c_i$ for each $i=1,\dots,k$, and thus there are $g'_i\in G$ such that $d(g'_iv,g_iv)<\delta$ and $d(g'_ic_0,c_i)<\epsilon/3$ for each $i$. Hence, if $\bar{g}'=(g'_1,\dots,g'_k)$, we have
$$ab\bar{\tau}(b)\equiv^{\epsilon/2} ac_0\bar{g}'(c_0)\ \text{ and }\ \bar{g}'\in \bar{g}G^k_{v,\delta}.$$

Finally, as $(G_{a,\epsilon/4}\cdot \bar g)_{\epsilon/4}$ is dense in $\bar g G_{v,\delta}^k$, there is $h\in G_{a,\epsilon/4}$ such that, for every $i=1,\dots,k$,
$$d(hg_ih\inv c_0,g'_ic_0)<\epsilon/2.$$
If we define $c=h\inv c_0$, using the previous inequality and the fact that $h\in G_{a,\epsilon/4}$, we get
$$d(ac_0\bar{g}'(c_0),h(a)h(c)h\bar{g}(c))<\epsilon/2,$$
and thus also
$$ab\bar{\tau}(b)\equiv^\epsilon ac\bar{g}(c),$$
as desired.
\end{proof}

\begin{examples}
\begin{enumerate}[wide]\ 
\item If $M$ is a countably infinite set with no further structure, the existentially closed automorphisms of $M$ are those permutations that have infinitely many $n$-cycles for every $n\geq 1$. There is a comeager conjugacy class in $\Sym(M)$, which consists of those permutations that, in addition to the previous property, have no infinite orbits. These correspond to the prime models of $T_1$.

\item Using Hrushovski's work \cite{hrushovski1992extending} on extending partial automorphisms of finite graphs and ideas by Kechris--Rosendal \cite{kecros}, it is easy to see that for all finite $k\geq 1$ the group of automorphisms of the random graph has $k$-generics, but it is well-known that the corresponding theory $T_k$ has no model companion for any $k\geq 1$.

\item If $M$ is a separable, atomless probability measure algebra, the existentially closed automorphisms of $M$ are those measure-preserving invertible transformations that are aperiodic, and they coincide with the metric generics of $\Aut(M)$. The conjugacy class of each single transformation is meager.
\end{enumerate}
\end{examples}

A discussion comparing the two notions of genericity for automorphisms of discrete, saturated, countable structures can be found in \cite{barbina2012generic}.

\noindent\hrulefill
\section{Main particular cases and equi-homogeneity}\label{s:particular-cases}

The conditions of exactness and of continuous symmetry of the independence relation, discussed in Section~\ref{s:def-of-indep}, always hold in the \emph{classical} first-order, $\aleph_0$-categorical, stable setting. Thus the main result of Section~\ref{s:main} specializes to the following theorem, which in the case $k=1$ is essentially contained in \cite{lasBeaux}, and in the general case had been noticed by A.\ Pillay and possibly others.

\begin{theorem}
Let $T$ be a classical, model complete, $\aleph_0$-categorical, superstable  theory in a countable language. Then the theory of $\mbF_k$-actions on models of $T$ has a model completion.
\end{theorem}

The example we are mostly interested in, on the other hand, is a metric one: that of probability measure algebras. We recall that the theory \APA\ of atomless probability algebras is $\aleph_0$-stable, $\aleph_0$-categorical, exact, and eliminates quantifiers in its natural language; see \cite[\textsection 16]{bbhu08}. Below we will show that its independence relation is continuously symmetric.

In fact, we will introduce a strong form of homogeneity which implies that the independence relation is continuously symmetric, and which we will verify in the separable model of \APA. Moreover, this homogeneity condition also holds for \emph{randomizations} of classical first-order theories, yielding an interesting class of metric structures to which our main result applies.

For the rest of this section $M$ will be a separable structure for a countable metric language.  We say $M$ is $\aleph_0$-categorical if its theory is $\aleph_0$-categorical.

To provide some motivation for our homogeneity condition, we recall first that Effros's Theorem implies that if $M$ is homogeneous, then the orbit maps $g\in G\mapsto ga\in Ga$ for $a\in M^n$ are open. (For background on the Effros Theorem, see for instance  \cite[\textsection 5]{ibameg} and the references therein.) Actually, as observed in \cite[Prop.~2.9]{benusvD-fin} and later also in \cite[Prop.~5.8]{ibameg}, we have the following stronger fact. We use the notation $G_{b,\epsilon}$ from the previous section (see equation (\ref{eq:basicneighb})), and for $a \in M^n$ and $\delta>0$ we define $B_\delta(a) = \{a' \in M^n \mid d(a,a')<\delta \}$.

\begin{prop}
Let $M$ be a homogeneous $\aleph_0$-categorical structure, with $G$ its automorphism group. Then for every $n,k\in\mbN$, 
$\epsilon>0$ and every $G$-orbit $O\subseteq M^n$ there is $\delta>0$ such that for all $a\in O$ and $b\in M^k$ we have $B_\delta(a)\cap O\subseteq G_{b,\epsilon}a$.
\end{prop}

The proposition states that the action of $G$ on each orbit $O$ satisfies the conclusion of Effros's Theorem in a uniform way, since the same $\delta$ works for every $a\in O$ (and simultaneously for all neighborhoods of the identity of the form $G_{b,\epsilon}$ with $b\in M^n$).  Later in this paper we will need an even stronger uniformity condition, namely that the same 
$\delta$ works for all $a\in M^n$ with no restriction to a given orbit.  

We will actually define two conditions, both requiring stronger uniformity than in the proposition.  For the second one we use the notation $G_\epsilon= \{g\in G:\sup_{c\in M}d(gc,c)<\epsilon\}$. That is, $G_\epsilon$ is the $\epsilon$-neighborhood of the identity for the uniform distance $\partial$ discussed in the previous section: $G_\epsilon=(1_G)_\epsilon$ (see equation (\ref{eq:unifneighb})).

\begin{defin} Let $M$ be a separable structure and $G$ its automorphism group.
\label{D:equi-homogeneous}
\begin{enumerate}
\item We say that $M$ is \emph{equi-homogeneous} if for every $n,k\in\mbN$ and $\epsilon>0$ there is $\delta>0$ such that for all $a\in M^n$ and $b\in M^k$ we have $B_\delta(a)\cap Ga\subseteq G_{b,\epsilon}a$.
\item We say that $M$ is \emph{uniformly equi-homogeneous} if for every $n\in\mbN$ and $\epsilon>0$ there is $\delta>0$ such that for all $a\in M^n$ we have $B_\delta(a)\cap Ga\subseteq G_\epsilon a$.
\end{enumerate}
\end{defin}

Clearly, if $M$ is uniformly equi-homogeneous then it is equi-homogeneous.  Further, by the converse of Effros's Theorem (i.e., if the orbit maps are open then the orbits are Polish, and hence closed since the action is isometric), every equi-homogeneous separable structure is homogeneous.

\begin{example}\label{ex:Urysohn}
The \emph{Urysohn sphere} is a uniformly equi-homogeneous, $\aleph_0$-categorical structure; see \cite[Prop.~5.3]{ibameg}. It is however not stable, so it does not fall within the hypotheses of our analysis of existentially closed actions. In fact, if Kikyo's result \cite[Thm.~3.3]{Ki} holds in the metric setting---which is unclear to us---one could use it to prove that $T_1$ has no model companion in this case.
\end{example}

\begin{lem}\label{l:from-homog-to-cont-sym}
Let $M$ be a stable, $\aleph_0$-categorical structure. If $M$ is equi-homogeneous, then the stable independence relation is continuously symmetric on $M$.
\end{lem}
\begin{proof}
This is just a consequence of the invariance of the independence relation. Indeed, assume $M$ is equi-homogeneous and let $\delta>0$ be such that for every $ab\in M^{n+k}$ and $c\in M^l$ we have $B_\delta(ab)\cap Gab\subseteq G_{c,\epsilon/2}ab$. Suppose we are given $b\in M^k$ with $b^\delta\ind_ac$, and let $b'\in M^k$ be such that $b'\eqstp_a b$, $b'\ind_ac$,  and $d(b,b')<\delta$. Then $ab'\in B_\delta(ab)\cap Gab$, so there is $g\in G$ such that $g(ab')=ab$ and $d(c,gc)<\epsilon/2$. Hence, by invariance, $b\ind_a gc$. We conclude that $b\ind_a c^\epsilon$ by Remark~\ref{rem:eps-ind-without-stp}.
\end{proof}

\subsection{Equi-homogeneity of the measure algebra}\label{ss:meas-alg}

In this subsection, $M$ denotes the probability measure algebra of an atomless Lebesgue space $(X,\mu)$ (i.e., a standard probability space), and $G$ its automorphism group. So $M$ is the unique separable model of \APA.  We recall that the canonical metric on $M$ is given by $d(a,b)=\mu(a\triangle b)$. Following our previous conventions, the finite powers $M^n$ are thus endowed with the distance $d(a,b)=\max_{i<n} \mu(a_i\triangle b_i)$.

However, a much better distance function is available in this case. Indeed, we may also endow the finite powers $M^n$ with the distance
$$d_P(a,b)\coloneqq \tfrac{1}{2}\sum_{s\in 2^n}\mu(p_s\triangle q_s),$$
where $(p_s)_{s\in 2^n}$ and $(q_s)_{s\in 2^n}$ are the partitions of $X$ generated by the $n$-tuples $a$ and $b$, respectively (defined up to measure zero). More precisely, if $a=(a_i)_{i<n}$, $b=(b_i)_{i<n}$, and if we denote 
$a_i^0=a_i$, $a_i^1=X\setminus a_i$ and $b_i^0=b_i$, $b_i^1=X\setminus a_i$, then for each $s\in 2^n=\{0,1\}^n$ we let
$$p_s=\bigcap_{i<n}a_i^{s(i)} \mbox{\ \ and\ \ } q_s=\bigcap_{i<n}b_i^{s(i)}.$$
It is clear that $d_P$ is a metric on $M^n$. It is moreover equivalent to the $\max$-distance $d$. Indeed, on the one hand, we have 
$$\mu(p_s\triangle q_s)\leq \mu\big(\bigcup_{i<n} (a_i^{s(i)}\triangle b_i^{s(i)})\big)\leq nd(a,b)$$
for every $s\in 2^n$. On the other hand, for every $i<n$ and $\eta\in\{0,1\}$,
$$d(a_i^\eta,b_i^\eta) \leq \mu\big(\bigcup_{s(i)=\eta} (p_s\triangle q_s)\big),$$
so that
$$2d(a_i,b_i) = d(a_i^0,b_i^0) + d(a_i^1,b_i^1)\leq 2d_P(a,b).$$
Thus,
\begin{equation}\label{eq:d_P-inequalities}
d(a,b)\leq d_P(a,b)\leq n2^{n-1}d(a,b).
\end{equation}
For elements $a,b\in M$ we have in particular that $d_P(a,b)=d(a,b)$. We also note that $d_P$, like $d$, is $G$-invariant, bounded by $1$, and monotone for concatenation of tuples, in the sense that $d_P(a,b)\leq d_P(aa',bb')$ for any $a,b\in M^n$ and $a',b'\in M^k$.

The distance $d_P$ is particularly nice because of the \emph{equality} in the following lemma, which is not necessary for the subsequent proposition and theorem (a much weaker inequality would suffice) but is crucial for a back-and-forth argument done in Section~\ref{s:profinite-completion}.

\begin{lem}\label{l:d_P-equality}
Let $a,b\in M^n$ be tuples with $a\equiv b$. Then there is $g\in G$ such that $ga=b$ and, for every finite tuple $c\in M^m$, we have
$$d_P(ac,bg(c))=d_P(a,b).$$
In particular, $\partial(g,1_G)\leq  d_P(a,b)$.
\end{lem}
\begin{proof}
Let as before $p=(p_s)_{s\in 2^n}$ and $q=(q_s)_{s\in 2^n}$ be the partitions induced by the tuples $a$ and $b$. As $a\equiv b$, we have $\mu(p_s)=\mu(q_s)$ for every $s\in 2^n$, and thus also $\mu(p_s\setminus q_s)=\mu(q_s\setminus p_s)$.

Now we define a map $g\colon X\to X$ in the following manner: for every $s\in 2^n$, let $g_s\colon X\to X$ be an invertible measure-preserving map sending $p_s\setminus q_s$ to $q_s\setminus p_s$. Then let $g\colon X\to X$ coincide with $g_s$ on each set $p_s\setminus q_s$, and be the identity elsewhere. Because $p$ and $q$ are partitions, $g$ is a well-defined, invertible measure-preserving map, which we may see as an element of $G=\Aut(M)$. Moreover, $g(p_s)=q_s$ for each $s$, so $ga=b$.

We check that $g$ has the property of the statement. If $c\in M$ is arbitrary, define $c_s^0=c\cap (p_s\setminus q_s)$ and $c_s^1=(X\setminus c)\cap (p_s\setminus q_s)$. Let $\tilde{p}$ and $\tilde{q}$ be the partitions generated by $ac$ and $bg(c)$, respectively. Given $s\in 2^n$, let $s0$ and $s1$ denote the two obvious extensions of $s$ to sequences in $2^{n+1}$. Then for $i\in\{0,1\}$ we have, by construction of $g$,
$$\tilde{p}_{si}\triangle \tilde{q}_{si} = c_s^i \cup gc_s^i,$$
and thus, since $c_s^0$, $c_s^1$, $gc_s^0$, and $gc_s^1$ are disjoint,
\begin{align*}
d_P(ac,bg(c)) & = \tfrac{1}{2}\sum_{r\in 2^{n+1}}\mu(\tilde{p}_r\triangle\tilde{q}_r) = \tfrac{1}{2}\sum_{s\in 2^n}\mu(c_s^0\cup gc_s^0)+\mu(c_s^1\cup gc_s^1) \\
& =\tfrac{1}{2}\sum_{s\in 2^n}\mu(c_s^0\cup c_s^1)+\mu(gc_s^0\cup gc_s^1) =\tfrac{1}{2}\sum_{s\in 2^n}\mu(p_s\triangle q_s) = d_P(a,b).
\end{align*}
In general, if $c\in M^m$, we define $c_s^i$ for each $i\in 2^m$ as the intersection of $(p_s\setminus q_s)$ with the $i$-th element of the partition generated by $c$, and a similar calculation yields the desired conclusion.

Finally, $\partial(g,1_G)=\sup_{c\in M}d(c,gc)\leq \sup_{c\in M}d_P(ac,bg(c))= d_P(a,b)$.
\end{proof}

\begin{prop}
The probability measure algebra $M$ is uniformly equi-homogeneous.
\end{prop}
\begin{proof}
Fix $n\in\mbN$ and $\epsilon>0$, and take $\delta=\epsilon/n2^{n-1}$. The previous lemma and the discussion preceding it show that we have $B_\delta(a)\cap Ga\subseteq G_\epsilon a$ for every $a\in M^n$. (In fact, it is easy to see that $\delta=\epsilon/n$ works as well; see the proof of Proposition~\ref{p:randomizations}.)
\end{proof}

Let us denote by \PMPk\ the theory of probability measure-preserving $\mbF_k$-actions, by which we mean the same thing as the theory $(\APA_\forall)_k$ of probability algebras with $k$ automorphisms and their inverses. We conclude the following.

\begin{theorem}\label{thm:main-example}
For every $k\leq\omega$, the theory \PMPk\ admits a model completion, \PMPFk*, which is complete, stable and has quantifier elimination.

In addition, if $A,B,C$ are sets in a model of \PMPFk*, then:
\begin{enumerate}
\item\label{item:real-alg-closure} The (real) algebraic closure of $A$ is the substructure $\langle A\rangle$ generated by $A$.
\item\label{item:indep-in-PMPk*} We have $A\ind^{\mathPMPFk*}_{C}B$ iff $\langle A\rangle\ind_{\langle C\rangle}\langle B\rangle$, where $\ind$ is the independence relation for probability algebras.
\end{enumerate}
\end{theorem}
\begin{proof}
The above results serve to verify the hypotheses of Theorem~\ref{thm:main} for $\APA_k$, and the conclusion applies also to \PMPk\ as per Remark~\ref{rem:T2forall=Tforall2}. We then use Corollary~\ref{cor:qe}: \PMPk\ has the amalgamation property, since any two pmp $\mbF_k$-systems $\mcX$ and $\mcY$ with a common factor $\mcZ$ can be merged together into their relatively independent joining, 
$\mcX\otimes_{\mcZ} \mcY$. (For a definition of the relatively independent joining in the standard separable setting see \cite[Ch.~6]{glaJoinings}; a definition for arbitrary systems is given, for instance, in \cite[\textsection 3]{ibatsaStrong}.) By considering the trivial factor we also see that \PMPk\ has the joint embedding property.
\end{proof}

In the special case $k=1$, as mentioned before, the existentially closed models are given by the aperiodic transformations, and moreover these are precisely the metrically generic automorphisms. These are consequences of Rokhlin's Lemma, which says that every aperiodic transformation is uniformly close to a cycle; see \cite[\textsection 18]{bbhu08}. The fact that the existentially closed transformations are metrically generic implies that \PMPF1* is \emph{approximately $\aleph_0$-categorical}: any two separable models of \PMPF1* can be conjugated arbitrarily close to one another, in the sense of the uniform metric $\partial$.

These properties of the model theory of $\mbZ$-actions on probability spaces can be generalized to actions of amenable groups by using the Ornstein--Weiss extension of Rokhlin's Lemma, as was observed a long while ago by the first and second authors in an unpublished note. Recently, Giraud \cite{girHyperfinite}, following results of Elek \cite{elekFinite}, further generalized this analysis to the case of hyperfinite pmp actions with a given IRS. In particular, the theory of pmp $\mbF_k$-actions with a fixed hyperfinite IRS is model complete and approximately $\aleph_0$-categorical.

\begin{question}
Is \PMPFk* approximately $\aleph_0$-categorical for $k>1$?
\end{question}

In \cite{BBM-topometric}, the authors show that there are metrically generic pmp $\mbF_k$-systems for every finite $k$. Thus, for such $k$ the previous question is equivalent to asking whether every model of \PMPFk* is metrically generic. A positive answer could be regarded as a form of Rokhlin's Lemma for the free group, in which the existentially closed systems (with their concrete description as given in Section~\ref{s:main}) take the role of the aperiodic transformations.

The paper \cite{BBM-topometric}, on the other hand, does not give a concrete example of a metrically generic system for $k>1$. In Section~\ref{s:profinite-completion} we will use and adapt their method to exhibit two concrete non-isomorphic, metrically generic pmp $\mbF_k$-actions, for arbitrary finite $k$.

\subsection[5B]{Actions on randomizations}

We end this section with the remark that the randomizations of classical, countably categorical structures are uniformly equi-homogeneous. We shall be brief and assume familiarity with randomizations of classical theories, in the sense of \cite{benkei}. For our purposes it is convenient to work with the one-sorted formalism, in which the \emph{event sort} is not included in the language; see \cite[Rmk.~2.12]{benkei}.
Let $M$ be a classical, countable $\mcL$-structure.  We fix an atomless Lebesgue space $\Omega=(X,\mu)$, and we denote by $M^\Omega$ the space of measurable functions from $X$ to $M$, identified up to measure zero. We endow $M^\Omega$ with the distance $d(f,f')=\mu\{x\in X:f(x)\neq f'(x)\}$, which is separable and complete. Then $M^\Omega$ is a metric structure in the language $\mcL^R$ containing a predicate $\mu\llbracket\varphi(x)\rrbracket$ for every $\mcL$-formula $\varphi(x)$, so that $\mu\llbracket\varphi(f)\rrbracket$ for $f\in (M^\Omega)^n$ is interpreted as the probability that the random variables $f$ satisfy $\varphi$. The \emph{randomized theory} $T^R=\Th(M^\Omega)$ always has quantifier elimination in the language $\mcL^R$. Moreover, if $T=\Th(M)$ is $\aleph_0$-categorical (respectively, $\aleph_0$-stable) then so is the theory $T^R$. See \cite{benkei} for these facts.

The automorphism group $\Aut(M^\Omega)$ can be naturally identified with the semidirect product $\Aut(M)^\Omega\rtimes\Aut(\Omega)$, where $\Aut(M)^\Omega$ is the group of measurable functions from $X$ to $\Aut(M)$ and $\Aut(\Omega)$ is the group of invertible measure-preserving transformations of $X$ (identified up to measure zero). See \cite{ibaRando}.

\begin{prop}\label{p:randomizations}
Let $M$ be a classical, $\aleph_0$-categorical structure. Then $M^\Omega$ is uniformly equi-homogeneous.
\end{prop}
\begin{proof}
Fix $n\in\mbN$ and $\epsilon>0$, and let $\delta=\epsilon/2n$. Take $f,f'\in (M^\Omega)^n$ with $f\equiv f'$ and $d(f,f')<\delta$ (recall that $d(f,f')=\max_{i<n}\mu\{x\in X:f_i(x)\neq f_i'(x)\}$). We want to find $g\in\Aut(M^\Omega)$ such that $gf=f'$ and $\partial(g,1_{\Aut(M^\Omega)})<\epsilon$.

We let $X_0=\{x\in X:f(x)\neq f'(x)\}$ (identifying $(M^\Omega)^n=(M^n)^\Omega$), which has measure $\mu(X_0)<\epsilon/2$. On the other hand, for every $p\in S_n(T)$ (where $T=\Th(M)$, and the type space is finite by categoricity), we let
$$X_p=\{x\in X:\tp(f(x))=p\}\cap X_0\text{ and }X'_p=\{x\in X:\tp(f'(x))=p\}\cap X_0.$$
Since $f\equiv f'$ and since $f$ and $f'$ coincide on $X\setminus X_0$, we have $\mu(X_p)=\mu(X'_p)$ for every $p\in S_n(T)$. Thus, we can choose $g_0\in\Aut(\Omega)$ such that $g_0(X_p)=X'_p$ (up to measure zero) for every $p\in S_n(T)$, and such that $g_0$ is the identity on $X\setminus X_0$. In particular, $\partial(g_0,1_{\Aut(M^\Omega)})<\epsilon/2$. Now, $\tp(g_0f(x))=\tp(f'(x))$ for every $x\in X$, and the images of $g_0f$ and $f'$ are countable. By homogeneity of $M$ we can easily find $g_1\in\Aut(M)^\Omega$ such that $g_1(x)(g_0f(x))=f'(x)$ for every $x\in X$, and moreover $g_1(x)=1_{\Aut(M)}$ for every $x\in X\setminus X_0$. In particular, $\partial(g_1,1_{\Aut(M^\Omega)})<\epsilon/2$. The automorphism $g=g_1g_0\in\Aut(M^\Omega)$ is then as desired.
\end{proof}

In particular, the randomization of a classical $\aleph_0$-categorical theory is $\aleph_0$-categorical and exact. Recalling that a classical, superstable, $\aleph_0$-categorical theory is $\aleph_0$-stable, we can conclude the following.

\begin{theorem}
Let $T$ be a classical, superstable, $\aleph_0$-categorical theory. Then for every $k\leq\omega$, the theory of $\mbF_k$-actions on models of the randomized theory $T^R$ has a model completion.
\end{theorem}

In \cite{BeZa} it is pointed out, for finite $k$, that if $g\in \Aut(M)^k$ is generic (i.e., its diagonal conjugacy class is comeager), if $c_g\in (\Aut(M)^k)^\Omega$ is the constant function with value $g$, and $t\in \Aut(\Omega)^k$ is metrically generic, then $(c_g,t)\in (\Aut(M)^k)^\Omega \rtimes\Aut(\Omega)^k\cong \Aut(M^\Omega)^k$ is metrically generic in the randomization. So when $M$ has generic $\mbF_k$-actions we can build models of $(T^R)_k^*$ in this fashion. It is natural to ask whether, in general, the models of $(T^R)_k^*$ are precisely the pairs $(h,t)$ such that $(M,h(x))\models T_k^*$ for almost every $x$ and $t$ induces a model of \PMPFk*.

\noindent\hrulefill

\section{Examples of generic actions on a Lebesgue space}\label{s:profinite-completion}

In this section we give a concrete example of a metrically generic action of any finitely generated free group on the probability measure algebra of a Lebesgue space (in the sense of Section~\ref{s:metric-generics}) and thus, in particular, of an existentially closed such action.

Recall that the profinite completion $\wFk$ of $\mbF_k$ is the inverse limit of the system of finite group quotients $\mbF_k\to\Gamma$, with morphisms given by surjective homomorphisms $\Gamma_1\to\Gamma_2$ commuting with the quotient maps. The limit $\wFk$ is a compact group containing a dense copy of $\mbF_k$, and is thus endowed with the action $\mbF_k\actson\wFk$ by left multiplication, which preserves the Haar measure $\mu$---the unique invariant probability measure on $\wFk$.

We will show that the pmp system $\mbF_k\actson(\wFk,\mu)$ yields a metrically generic measure-preserving action of $\mbF_k$, for every $k<\omega$. Note that this is a significant strengthening of a result of Kechris \cite{kechrisWeak}, who showed that the diagonal conjugacy class of the system $\mbF_k\actson (\wFk,\mu)$ is \emph{dense} in $\Aut(\wFk,\mu)^k$. We will use Kechris's result in our proof. We are grateful to Todor Tsankov for suggesting  that we consider this example.

We will in fact provide two non-isomorphic examples of metrically generic actions, the other being the product of the action on $\wFk$ with the trivial action on an atomless Lebesgue space. The latter is more directly amenable to our method of proof, so we will consider it first.

A general strategy to produce metrically generic actions is the method of \emph{countable approximating substructures} developed by Ben Yaacov, Melleray and the first named author in \cite{BBM-topometric}. This method permits proving that a given action on a separable metric structure $M$ is metrically generic by showing that it restricts to a \emph{generic} action on an appropriate classical, countable structure $N$ that can be seen as a discrete version of $M$.

More precisely, if $M$ is a separable metric structure and $N$ is some classical, countable structure, we say, following \cite[Def.~5.3]{BBM-topometric}, that $N$ is a \emph{countable approximating substructure of $M$} if:
\begin{enumerate}
\item the domain of $N$ is a dense countable subset of $M$;
\item every automorphism of $N$ extends to an automorphism of $M$.
\end{enumerate}
Note that, under the first condition, an extension of an automorphism of $N$ to an automorphism of $M$ is necessarily unique. Thus if $N$ is a countable approximating substructure of $M$, we can see $H=\Aut(N)$ as a subgroup of $G=\Aut(M)$ (although the topology on $H$ is not the restriction of that of $G$). We say, furthermore, that $N$ is a \emph{good countable approximating substructure of $M$} if, in addition:
\begin{enumerate}
\item[(3)] for every $\epsilon>0$ and every subset $U\subseteq\Aut(N)$ that is open in the topology of $\Aut(N)$, the $\epsilon$-fattening $(U)_\epsilon\subseteq \Aut(M)$ with respect to the uniform metric of $\Aut(M)$ is open in the topology of $\Aut(M)$.
\end{enumerate}

To establish the last condition it is in fact enough to check that for every open neighborhood $V$ of the identity in $H$ the set $(V)_\epsilon$ contains an open neighborhood of the identity in $G$. Indeed, if then we are given an open subset $U\subseteq H$ and an arbitrary element $g\in (U)_\epsilon$, we can find $h\in U$, $\delta>0$,  and an open neighborhood of the identity $V\subseteq H$ such that $\partial(g,h)+\delta<\epsilon$ and $hV\subseteq U$. Then $g(V)_\delta$ is a neighborhood of $g$ in $G$ and $g(V)_\delta\subseteq (hV)_\epsilon\subseteq (U)_\epsilon$, showing that $g$ is in the interior of $(U)_\epsilon$.

On the other hand, since $N$ is a classical structure, a base of open neighborhoods of the identity in $H$ is given by the stabilizers $H_a=\{h\in H:ha=a\}$, ranging over the finite tuples $a\in N^n$. Thus, condition (3) above is equivalent to the following:
\begin{enumerate}
\item[(3')] for every $\epsilon>0$ and finite tuple $a\in N^n$, the set $(H_a)_\epsilon$ contains an open neighborhood of the identity in $G$.
\end{enumerate}

Now let us recall that, following \cite{trussGeneric,kecros,ivaGeneric} and others, a tuple of automorphisms $\bar{g}\in H^k$ (and the action $\mbF_k\actson N$ induced by it) is called \emph{generic} if the diagonal conjugacy class ${H\cdot\bar{g}}$ is comeager in $H^k$. This is indeed a particular case of the notion of metric generics discussed in Section~\ref{s:metric-generics}, because the uniform metric on the automorphism group of a classical structure is just the discrete metric.

Theorem~5.6 of \cite{BBM-topometric} then says the following.

\begin{theorem}\label{thm:good-app-substructures}
Suppose $N$ is a good countable approximating substructure of a separable metric structure $M$. If $\bar{g}\in \Aut(N)^k$ is generic, then its action on $M$ is metrically generic.
\end{theorem}

In our case, the metric structure $M$ that we wish to consider is the probability measure algebra of an atomless Lebesgue space $(X,\mu)$, which for now it is convenient to see as the unit interval $X=\mbI=[0,1]$ with the Lebesgue measure.  The approximating substructure, in turn, will be the (measured) Boolean algebra $N$ generated within $M$ by the \emph{rational subintervals} of $\mbI=[0,1]$. As earlier, we denote by $G$ and $H$ the automorphism groups of $M$ and $N$, respectively.  (Later we will rather see $(X,\mu)$ as the profinite completion $\wFk$ with its Haar measure, or as the product $\wFk\otimes \mbI$.)

Before arguing that $N$ is a good approximating substructure of $M$, let us say a few words about $N$ itself. We see $N$ as a classical structure in the language $\{\cap,\cup,\triangle,{\bf 0},{\bf 1}\}$ of Boolean algebras expanded with a (countable) family of unary predicates $\{\mu_r\}_{r\in\mbQ\cap [0,1]}$, so that $N\models \mu_r(a)$ iff $\mu(a)=r$. The structure $N$ is in fact the \emph{Fraïssé limit} of the class of finite probability measure algebras with rational-valued measures; see \cite[Prop.~2.4]{kecros}. Kechris and Rosendal denote this Fraïssé class by $\MBAQ$. In particular, $N$ is ultrahomogeneous, so every two partitions of $[0,1]$ by elements of $N$ with same corresponding measures can be mapped one to another by some $h\in H$.

The domain of $N$ is a dense subset of $M$. Moreover, for any tuple $a\in M^n$ such that the elements of the algebra generated by $a$ have \emph{rational measures}, and for any $\epsilon>0$, one can produce $b\in N^n$ such that $d(a,b)<\epsilon$ and $a\equiv b$. It is as well clear that every automorphism of $N$ extends to a unique automorphism of $M$. By an abuse of notation, we will identify $H$ with the subgroup of $G$ given by all these extensions, but keeping in mind that the topologies are not the same.

To see that $N$ is a good approximating substructure of $M$, we are left to show that for every $\epsilon>0$ and every stabilizer subgroup $H_a=\{h\in H:ha=a\}$, $a\in N^n$, the $\epsilon$-fattening $(H_a)_\epsilon$ is open in $G$ (or simply, that it is a neighborhood of the identity). There is no difference, as for checking this condition, between our case and the approximating substructure considered in \cite[\textsection 6.3]{BBM-topometric}, namely the Boolean algebra generated by dyadic intervals (and for their purposes, they could as well have considered the one generated by rational intervals). However, as all details are omitted there, we give a complete argument here.

We begin by noting that Lemma~\ref{l:d_P-equality} holds also for $N$ and $H$.

\begin{lem}\label{l:d_P-equality-for-N}
Let $a,b\in N^n$ be tuples with $a\equiv b$. Then there is $h\in H$ such that $ha=b$ and, for every $c\in N$, we have $d_P(ac,bh(c))=d_P(a,b)$.

As a consequence, if we are given $b\in N^n$, $c\in N^n$ and $g\in G$ such that $b\equiv c$ and $d_P(gb,c)<r$ for some $r>0$, then for every $b'\in N$ there is $c'\in N$ such that $bb'\equiv cc'$ and $d_P(g(bb'),cc')<r$.
\end{lem}
\begin{proof}
The main assertion is proved exactly as in Lemma~\ref{l:d_P-equality}, \emph{mutatis mutandis}, using the ultrahomogeneity of $N$.

For the second part, let $b,c,g,r,b'$ be as in the statement. Choose by density some $e\in N^n$ and $e'\in N$ such that $d_P(ee',g(bb'))<\frac{1}{2}(r-d_P(gb,c))$ and $ee'\equiv g(bb')$, which is possible because the algebra generated by $g(bb')$ has rational measures. In particular, $d_P(e,gb)<\frac{1}{2}(r-d_P(gb,c))$ and $e\equiv c$. By the main assertion, there is $h\in H$ such that $h(e)=c$ and $d_P(ee',ch(e'))=d_P(e,c)$. Thus, if we let $c'=h(e')\in N$, we have $bb'\equiv ee'\equiv cc'$ and
$$d_P(g(bb'),cc')\leq d_P(g(bb'),ee')+d_P(e,c)\leq d_P(g(bb'),ee')+d_P(e,gb) +d_P(gb,c)<r,$$
as desired.
\end{proof}

We deduce the following.

\begin{prop}\label{p:H-unif-dense-in-G}
The group $H$ is uniformly dense in $G$. More generally, $\ov{H_a}^\partial=G_a$ for every finite tuple $a\in N^n$.
\end{prop}
\begin{proof}
Let $a\in N^n$, $g\in G_a$, and $\epsilon>0$, and consider the family
$$F=\{(b,c)\in N^m\times N^m: m\in\mbN,b\equiv_a c,d_P(ag(b),ac)<\epsilon\}.$$
We may see a pair $(b,c)\in F$ as a partial isomorphism $p$ of $N$ that sends $b$ to $c$ and fixes~$a$, and satisfies the condition $d_P(ag(b),ap(b))<\epsilon$. We claim that this family of partial automorphisms has the back-and-forth property. Indeed, this follows immediately from the second part of the previous lemma.

Therefore we can construct an automorphism $h\in H$ such that $ha=a$ and $d_P(hb,gb)\leq\epsilon$ for every $b\in N$, and we conclude that $H_a$ is uniformly dense in $G_a$.
\end{proof}

Given $a\in M^n$ and $\epsilon>0$, let us consider the sets
$$G_{a,\epsilon}^P = \{g\in G: d_P(ga,a)<\epsilon\},$$
which form a basis for the Polish topology of $G$.

\begin{cor}\label{c:eps-fatts-are-open}
For every $a\in N^n$ and $\epsilon>0$ we have $G_{a,\epsilon}^P\subseteq (G_a)_\epsilon = (H_a)_\epsilon$.

The $\epsilon$-fattening of any open subset $U\subseteq H$ is open in $G$.
\end{cor}
\begin{proof}
By the previous proposition we have $(G_a)_\epsilon = (H_a)_\epsilon$. Now let $g\in G_{a,\epsilon}^P$ and consider $a'=ga$. By Lemma~\ref{l:d_P-equality}, there is $h\in G$ such that $ha'=a$ and $\partial(h,1_G)\leq d_P(a',a)<\epsilon$. Hence $hg\in G_a$ and $\partial(hg,g)=\partial(h,1_G)<\epsilon$, showing that $g\in (G_a)_\epsilon$.

The rest follows from the equivalence between the conditions (3) and (3') in the definition of a good countable approximating substructure.
\end{proof}

We have proved:

\begin{prop}\label{p:N-goog-app-M}
The measured Boolean algebra $N$ is a good countable approximating substructure of the measure algebra $M$.
\end{prop}

Kechris and Rosendal proved by abstract means that there is a generic action $\mbF_k\actson N$; see \cite[Thm.~6.5]{kecros}. In the case of one automorphism ($k=1$), it follows from Akin \cite[Thm.~4.17]{akinGood} that the generic action is the one induced by the product of the trivial (i.e., identity) action and the universal adding machine (i.e., the map $x\mapsto x+1$ in the profinite completion $\widehat{\mbZ}$). We next give an explicit description of the generic action $\mbF_k\actson N$ that generalizes Akin's result.

Let $N_0=\Clop(\wFk)$ be the Boolean algebra of clopen subsets of the profinite completion of $\mbF_k$. We endow $N_0$ with the measure $\mu$ induced by the Haar measure of $\wFk$. The elements of $N_0$ are of the form $a_{\pi,S}=\pi\inv(S)$ where $\pi\colon\wFk\to \Gamma$ is a continuous homomorphism onto a finite group and $S\subseteq\Gamma$. It is easy to show from this description that 
$$\{\mu(a):a\in N_0\}=\mbQ\cap [0,1].$$
Moreover, $\mu$ is a \emph{good measure} on $N_0$, meaning that for every $a \in N_0$ and $s \in \mbQ\cap [0,\mu(a)]$ there exists $b \in N_0$ such that $b \subseteq a$ and $\mu(b)=s$.\footnote{Indeed, say $a=a_{\pi,S}$, $|S|=n$, $|\Gamma|=k$ (so that $\mu(a)=n/k$) and write $\Gamma=\mbF_k/K$ for some $K\trianglelefteq \mbF_k$. Take $s\leq\mu(a)$, $s=m/p$. Choose $L\trianglelefteq\mbF_k$ of finite index such that $np$ divides $[\mbF_k:L]$. Now consider $K' = K\cap L$. Then $[\mbF_k:L]$ (and $np$) divides $[\mbF_k:K']$. On the other hand, $\pi$ factors through $\pi'\colon\mbF_k\to \mbF_k/K'$, so inside the preimage $\pi\inv(\gamma)$ of any $\gamma\in\Gamma$ we can find a preimage of a subset of $\mbF_k/K'$ of size $m[\mbF_k:K']/np$, which is thus of measure $m/np$ (note that $m/np \leq 1/k$ as $s\leq \mu(a)$). By doing this for the preimage of each $\gamma\in S$ and taking the union, we obtain a clopen set $b\subseteq a$ with $\mu(b)=s$.}

Now, there is up to isomorphism just one countable measured Boolean algebra with these properties, which is indeed the Fraïssé limit of $\MBAQ$ (see, for instance, \cite[\textsection 6]{ibamelFull}).  Hence $N_0$ is isomorphic to $N$, and so is the product algebra $N_0\otimes N_1$ (with the product measure), where $N_1$ is yet another copy of $N$.

In what follows we identify $N=N_0\otimes N_1$, and thus also $H=\Aut(N_0\otimes N_1)$. The canonical left action $\mbF_k\actson\wFk$ on the profinite completion induces a dual action $\mbF_k\actson N_0$, and thus an action $\mbF_k\actson N$ that fixes every element of the subalgebra $\{{\bf 0},{\bf 1}\}\otimes N_1\subseteq N$.

We recall some basic facts and introduce some convenient notations:
\begin{itemize}[leftmargin=12pt]
\item Given a continuous homomorphism $\pi\colon\wFk\to \Gamma$ onto a finite group and $S\subseteq\Gamma$, we write $a_{\pi,S}=\pi\inv(S)$ as above. Let us denote by $A_\pi\coloneqq \{a_{\pi,S}\}_{S\subseteq\Gamma}$ the finite subalgebra of $N_0$ induced by $\pi$; its atoms are the elements $a_{\pi,\gamma}=\pi\inv(\gamma)$ for $\gamma\in\Gamma$. The group $\Gamma$ acts on $A_\pi$ by automorphisms, by $\gamma a_{\pi,S}=a_{\pi,\gamma S}$. Then the action $\mbF_k\actson N_0$ can be described as follows: if $a\in N_0$ belongs to $A_\pi$ for some $\pi$ as above, then $wa=\pi(w)a$ for every $w\in\mbF_k$; this does not depend on the choice of $\pi$.

\item The finite subalgebras $A_\pi\otimes B\subseteq N$ (for $\pi$ some finite quotient of $\wFk$ and $B$ some finite subalgebra of $N_1$) are invariant for the action $\mbF_k\actson N$, and every finite subset of $N$ is contained in some such subalgebra.

\item If $A\in\MBAQ$ and $\mbF_k\actson A$ is an ergodic action (i.e., one which is transitive on the set $X_A$ of atoms of $A$), we have a homomorphism $\mbF_k\to\Gamma$ onto a finite group $\Gamma\leq\Sym(X_A)$ (inducing an action $\mbF_k\actson\Gamma$) and a surjective $\mbF_k$-equivariant map $\Gamma\to X_A$, $\gamma\mapsto \gamma x_0$ (where $x_0\in X_A$ is any  atom of $A$). Hence we also have a surjective $\mbF_k$-equivariant map $\wFk\to X_A$, which is moreover continuous and measure-preserving. We deduce that the action $\mbF_k\actson N_0$ contains an isomorphic copy of every ergodic action $\mbF_k\actson A$ with $A\in\MBAQ$. In turn, as it follows easily, the action $\mbF_k\actson N$ contains a copy of \emph{every} action $\mbF_k\actson A$, $A\in\MBAQ$.

\item As is easy to see, any tuple of \emph{partial} automorphisms of a given algebra $A\in\MBAQ$ can be extended to a tuple of automorphisms of some larger finite algebra $B\in\MBAQ$. Moreover, one may assume that all atoms of $B$ have equal measure. This is called the \emph{Hrushovski property} of $\MBAQ$ in \cite[\textsection 6.2]{kecros}, where a proof is given, in the discussion before the statement of Theorem 6.5, page 333.
\end{itemize}

\begin{theorem}\label{thm:Fk-on-N-is-generic}
The action $\mbF_k\actson N_0 \otimes N_1 \cong N$ defined above is generic.
\end{theorem}
\begin{proof}
Recall that $N_0$ is the Boolean algebra of clopen subsets of the profinite completion $\wFk$
of $\mbF_k$ equipped with the restriction of the Haar measure, and $N_1 \cong N$.  The action in question is the product of the dual action of $\mbF_k$ on $N_0$ induced by the canonical left action $\mbF_k\actson\wFk$ and the identity action of $\mbF_k$ on $N_1$.

We fix a $k$-tuple $\bar{w}$ of free generators of $\mbF_k$. Let $\bar{g}\in H^k$ be the tuple of automorphisms of $N$ corresponding to $\bar{w}$ through the action $\mbF_k\actson N$, and let $\bar{p}\in\Aut(N_0)$ be the one obtained through the action $\mbF_k\actson N_0$, so that for every $a\in N_0$, $b\in N_1$, and $i=1,\dots, k$ we have $g_i(a\otimes b)=p_i(a)\otimes b$. We wish to show that the conjugacy class $H\cdot\bar{g}$ is comeager in~$H^k$.

We first check that it is dense. Take $\bar{f}\in H^k$ and $a\in N^n$, and let us find $h\in H$ such that $(h\cdot\bar{g})(a)= \bar{f}(a)$. By the Hrushovski property, there a finite algebra $A$ containing $a$ and $\bar{f}(a)$ together with an action $\mbF_k\actson A$ by automorphisms of $A$ that extends $\bar{f}$ on~$a$. Since our action $\mbF_k\actson N$ contains a copy of this finite action, there is, by ultrahomogeneity, some $h\in H$ such that $(h\cdot\bar{g})(a)=\bar{f}(a)$, as desired.

Now, by the general theory of Polish group actions (see \cite[\textsection 3]{kecros}, or \cite[Lem.~5.5]{ibameg}), and since $H\cdot\bar{g}$ is dense, showing that $H\cdot\bar{g}$ is comeager is equivalent to showing that for every neighborhood $U\subseteq H$ of the identity, $\bar{g}$ belongs to the interior of the closure $\ov{U\cdot\bar{g}}$. Therefore it is enough to prove the following: for every finite quotient $\pi\colon\wFk\to\Gamma$ and every finite subalgebra $B\subseteq N_1$, we have
$$\bar{g}H_{A_\pi\otimes B}^k\subseteq \ov{H_{A_\pi\otimes B}\cdot\bar{g}},$$
where $H_{A_\pi\otimes B}$ is the pointwise stabilizer of the subalgebra $A_\pi\otimes B\subseteq N$. We will identify $\Gamma$ with the corresponding subgroup of $\Aut(A_\pi)$.

Take then $\bar{f}\in \bar{g}H_{A_\pi\otimes B}^k$, that is, $\bar{f}\in H^k$ such that $\bar{f}|_{A_\pi\otimes B}=\bar{g}|_{A_\pi\otimes B}$. Let $C\subseteq N$ be a finite subalgebra extending $A_\pi\otimes B$. We want to find $h\in H_{A_\pi\otimes B}$ such that $\bar{f}|_C=(h\cdot\bar{g})|_C$. Again by the Hrushovski property (and ultrahomogeneity), up to modifying $\bar{f}$ outside $C$ and extending $C$, we may assume that $\bar{f}$ induces an action $\mbF_k\actson C$ by automorphisms. Let $\Gamma'$ be the group generated by $\bar{f}$ inside $\Aut(C)$, and let $\pi'\colon\wFk \to \Gamma'$ be the continuous homomorphism that maps $\bar{w}$ to $\bar{f}|_C$. Note that $\pi$ factors through $\pi'$, i.e., $\pi=\nu\circ\pi'$ where $\nu\colon \Gamma'\to\Gamma$ is the homomorphism that sends $\bar{f}|_C$ to $\bar{p}|_{A_\pi}$. Thus, $A_{\pi'}\supseteq A_\pi$.

Next we construct a finite extension $B'\supseteq B$. For each atom $c\in C$, let $o(c)\in N$ be the union of the translates of $c$ by $\Gamma'$. Note that if $b$ is the unique atom of $B$ such that $c\subseteq {\bf 1}\otimes b$, then also $o(c)\subseteq {\bf 1}\otimes b$, because $\Gamma'$ acts trivially on $B$. We consider the set
$$O=\{o(c):c\text{ is an atom of }C\}.$$
For each $o\in O$, let $b(o)$ be the unique atom of $B$ such that $o\subseteq b(o)$. Then choose a family $\{b_o\}_{o\in O}$ of disjoint elements of $B$ such that $b_o\subseteq b(o)$ and $\mu(b_{o})=\mu(o)$ for every $o\in O$. In other words, we choose $\{b_o\}_{o\in O}$ a partition of the unit of $N_1$ isomorphic to $O$ over $B\cong \{{\bf 0},{\bf 1}\}\otimes B$. We denote by $B'$ the algebra generated by this partition inside $N_1$. Clearly, $B'\supseteq B$.

Now we construct a subalgebra $C'\subseteq A_{\pi'}\otimes B'$ containing $A_\pi\otimes B$, together with an isomorphism $\sigma\colon C\to C'$ that acts as the identity on $A_\pi\otimes B$. Given $o\in O$, let $P(o)$ be the set of atoms $c\in C$ contained in $o$ (i.e., $P(o)$ is the $\Gamma'$-orbit of any $c$ with $o=o(c)$). For each $o\in O$, we choose some $c_0\in P(o)$ such that $c_o\subseteq a_{\pi,1_\Gamma}\otimes {\bf 1}$, where $1_\Gamma\in\Gamma$ is the identity element. On the other hand, given $c\in P(o)$, we set
$$S_c\coloneqq \{\gamma'\in\Gamma': \gamma' c_o=c\},$$
and
$$\sigma(c)\coloneqq a_{\pi',S_c}\otimes b_o \in A_{\pi'}\otimes B'\subseteq N.$$
Let $C'$ be the algebra generated by the elements $\sigma(c)$ with $c$ ranging over the atoms of $C$. Observe that, if $c,c'$ are disjoint atoms of $C$, then:
\begin{itemize}[leftmargin=20pt]
\item if $o(c)\neq o(c')$, then $b_{o(c)}$ and $b_{o(c')}$ are disjoint in $N_1$;
\item if $o(c)= o(c')$, then $S_c\cap S_{c'}=\emptyset$, so $a_{\pi',S_c}$ and $a_{\pi',S_{c'}}$ are disjoint in $N_0$;
\end{itemize}
In any case, $\sigma(c)$ and $\sigma(c')$ are disjoint in $N$. Now let $c\in C$ be an atom and $o=o(c)$. We have:
$$\mu(\sigma(c))=\mu(a_{\pi',S_c})\mu(b_o)=\frac{|S_c|}{|\Gamma'|}\mu(o)=\frac{|S_c|}{|\Gamma'|}|P(o)|\mu(c)=\mu(c).$$
We deduce that the elements $\sigma(c)$ are the atoms of $C'$, and that the map $c\mapsto\sigma(c)$ induces an isomorphism $\sigma\colon C\to C'$.

We argue that $\sigma$ is the identity on $A_\pi\otimes B$. Indeed, if $a_{\pi,\gamma}\otimes b$ is an atom of $A_\pi\otimes B$ containing some atom $c$ of $C$, then:
\begin{itemize}[leftmargin=20pt]
\item $a_{\pi',S_c}\subseteq a_{\pi,{\gamma}}$, because if $\gamma'\in S_c$ then $\gamma'$ necessarily maps $a_{\pi,{1_\Gamma}}$ to $a_{\pi,{\gamma}}$, so $\nu(\gamma')=\gamma$;
\item $b_o\subseteq b$, by construction.
\end{itemize}
So $a_{\pi,\gamma}\otimes b$ contains $\sigma(c)$ as well.

Finally, we check that $\sigma\bar{f}=\bar{g}\sigma$ on $C$. By construction, for each $i=1,\dots,k$ and every atom $c\in C$, we have
$$p_i(a_{\pi',S_c})=a_{\pi',\pi'(w_i)S_c}=a_{\pi',S_{f_i(c)}},$$
and thus
$$g_i(\sigma(c)) = p_i(a_{\pi',S_c})\otimes b_{o(c)} = a_{\pi',S_{f_i(c)}}\otimes b_{o(c)}=\sigma(f_i(c)),$$
as desired. By homogeneity, there is $h\in H$ extending the inverse $\sigma\inv\colon C'\to C$. Since $\sigma$ acts as the identity on $A_\pi \otimes B$, so does $h$, and it follows that $h\in H_{A_\pi\otimes B}$ and $\bar{f}|_C = (h\cdot \bar{g})|_C$, which concludes the proof.
\end{proof}

We are now able to exhibit an example of a metrically generic action on the probability measure algebra of a Lebesgue space. Indeed, by Proposition~\ref{p:N-goog-app-M} and Theorem~\ref{thm:good-app-substructures}, the action $\mbF_k\actson M$ induced by the generic action $\mbF_k\actson N$ is metrically generic. Note that the natural presentation of the action on $M$ induced by the action on $N$ discussed above is the dual of the pmp system
$$\mbF_k\actson \wFk\otimes \mbI,$$
where $\mbF_k$ acts on the profinite completion $\wFk$ by left multiplication, and trivially on $\mbI$.

\begin{theorem}\label{thm:metric-gen-profinite-times-trivial}
The pmp system $\mbF_k\actson \wFk\otimes\mbI$ is metrically generic, and in particular existentially closed.
\end{theorem}

The particular properties of this example enable us to draw a conclusion about the model completion \PMPFk*. Indeed, the pmp system $\mbF_k\actson \wFk\otimes\mbI$ is profinite, and in particular \emph{compact} (a.k.a.\ \emph{isometric} or \emph{discrete spectrum}). By contrast, we have the following result of Kerr--Pichot \cite{KerrPichot08}; see also  \cite[Thm.~12.9]{kechrisGlobal}.  Recall that by $(X,\mu)$ we mean any atomless Lebesgue space.

\begin{theorem}
Weakly mixing pmp actions of $\mbF_k$ form a dense $G_\delta$ subset of $\Aut(X,\mu)^k$.
\end{theorem}

It follows from this theorem that the class of metrically generic pmp actions of the free group contains a weakly mixing system. Now, weakly mixing and compact systems are \emph{disjoint}, and in particular share no common factors (see \cite[Thm.~6.27]{glaJoinings}). Thus, we have two models of \PMPFk* that cannot simultaneously realize any non-trivial complete type.

\begin{cor}
The theory \PMPFk* has no isolated types over the empty set, other than the trivial ones.
\end{cor}

The rest of the section is devoted to showing that the pmp system $\mbF_k\actson \wFk$ is metrically generic as well. It is tempting to try to deduce it from the metric genericity of the system $\mbF_k\actson \wFk\otimes\mbI$, but we are not sure how to do so; see Corollary~\ref{c:profinite-vs-profinite-times-trivial} and Question~\ref{q:profinite-vs-profinite-times-trivial} at the end of the section. Instead, we will adapt our previous argument step-by-step to make it work for the profinite completion alone.

Let us begin by recalling Kechris's result mentioned earlier.

\begin{theorem}\label{thm:kechris}
The diagonal conjugacy class of the action $\mbF_k\actson \wFk$ is dense in $\Aut(\wFk,\mu)^k$.
\end{theorem}
\begin{proof}
See \cite[\textsection 1-3]{kechrisWeak}, and particularly Theorem~3.1.
\end{proof}

The discrete profinite action $\mbF_k\actson N_0=\Clop(\wFk)$, on the other hand, has no non-trivial fixed points, so its diagonal conjugacy class is not dense in $\Aut(N_0)^k$. We may, however, consider the set
$$E_k=\bigcap_{a\in N_0\setminus\{{\bf 0},{\bf 1}\}}\bigcup_{1\leq i\leq k}\left\{\bar{h}\in \Aut(N_0)^k:h_i(a)\neq a\right\},$$
consisting precisely of those $\bar{h}$ whose induced action on $N_0$ has no non-trivial fixed points. The set $E_k$ is closed in $\Aut(N_0)^k$, so in particular a Polish subspace. Our strategy will consist in showing that the diagonal conjugacy class of the action $\mbF_k\actson N_0$ is comeager in $E_k$, and then use the following variant of Theorem~\ref{thm:good-app-substructures}.

Given a classical structure $N$, a tuple $\bar{p}\in\Aut(N)^k$ and a Polish subspace $E\subseteq\Aut(N)^k$, let us say that $\bar{p}$ (or the associated action on $N$) is \emph{generic in $E$} if the diagonal conjugacy class $\Aut(N)\cdot\bar{p}$ is a comeager subset of $E$.

\begin{theorem}\label{thm:good-app-substructures-variant}
Suppose $N$ is a good countable approximating substructure of a separable metric structure $M$. Let $k\in\mbN$ and let $E\subseteq \Aut(N)^k$ be a Polish subspace with the following properties:
\begin{itemize}[leftmargin=20pt]
\item $E$ is dense as a subset of $\Aut(M)^k$;
\item for every $\epsilon>0$ and every relatively open subset $U\subseteq E$, the $\epsilon$-fattening $(U)_\epsilon$ is open in $\Aut(M)^k$.
\end{itemize}
Let $\bar{p}\in\Aut(N)^k$. If $\bar{p}$ is generic in $E$, then its action on $M$ is metrically generic.
\end{theorem}
\begin{proof}
Follows immediately from \cite[Thm.~5.2]{BBM-topometric} and the subsequent remarks.
\end{proof}

Our goal in the rest of this section is to apply the preceding theorem to the case where $M$ is the measure algebra of $(\wFk,\mu)$, $N$ is $N_0$, $E$ is the set $E_k$ defined above, and $\bar p$ corresponds to the generators of the action $\mbF_k\actson N_0$ .  Theorem~\ref{thm:kechris} yields that $E_k$ is dense when seen as a subset of $\Aut(\wFk,\mu)^k$. Next we check that  the second condition of the preceding theorem is also satisfied (see Corollary \ref{c:eps-fatts-are-open-variant}).  After doing that, we show that $\bar p$ is generic in $E_k$ (see Theorem \ref{thm:Fk-on-N0-is-generic-in-Ek}).

In what follows we take $H=\Aut(N_0)$ and $G=\Aut(\wFk,\mu)$, with the usual identification of $H$ as a subset of $G$. As before, given a subset or tuple $A\subseteq N_0$, by $H_A$ we denote the pointwise stabilizer of $A$ in $H$, and $H_A^k$ denotes $(H_A)^k$. Similarly for $G_A$ and $G_A^k$.

\begin{lem}\label{l:small-moving}
Given $\delta>0$ and finite subalgebras $C\subseteq C'\subseteq N_0$, there exists $s\in H_C$ such that $0<\mu(s(c)\triangle c)<\delta$ for every $c\in C'\setminus C$.
\end{lem}
\begin{proof}
If $C$ is the trivial subalgebra then this is easy; for instance, we may see $N_0$ as the algebra of rational subintervals of $\mbI$ and choose $s$ to be a sufficiently small, non-trivial rational translation modulo~1. If $C$ is an arbitrary finite subalgebra, it is enough to do a similar construction inside each atom of $C$.
\end{proof}

\begin{prop}\label{p:Ek-unif-dense-in-G}
The set $E_k$ is uniformly dense in $G^k$. More generally,
$$\ov{\bar{g}H^k_a\cap E_k}^\partial=\bar{g}G_a^k$$
for every $\bar{g}\in E_k$ and every finite tuple $a\in N_0^n$.
\end{prop}
\begin{proof}
Let $\bar{g}\in E_k$ and $a\in N_0^n$. It follows from Proposition~\ref{p:H-unif-dense-in-G} that $\ov{\bar{g}H_a^k}^\partial = \bar{g}G_a^k$, so it suffices to show that $\bar{g}H_a^k\subseteq \ov{\bar{g}H^k_a\cap E_k}^\partial$. Now, given $\bar{h}\in \bar{g}H_a^k$ and $\epsilon>0$, we may modify $\bar{h}$ slightly (and in fact, just one of the automorphisms in the tuple) to obtain $\bar{h}'\in \bar{g}H_a^k\cap E_k$ with $\partial(\bar{h}',\bar{h})\leq\epsilon$.

To see this, let $A\subseteq N_0$ be the subalgebra generated by $a$, and let $F$ be the family of partial automorphisms $f\colon B\to C$ defined on finite subalgebras of $N_0$ and satisfying:
\begin{itemize}[leftmargin=20pt]
\item $A\subseteq B$;
\item $f|_A=h_1|_A$;
\item $d_P(f(\bar{B}),h_1(\bar{B}))<\epsilon$, where $\bar{B}$ is some enumeration of $B$;
\item $f(b)\neq b$ for every $b\in B\setminus A$.
\end{itemize}
We argue that $F$ has the back-and-forth property. Suppose $B'\subseteq N_0$ is a finite subalgebra with $B\subseteq B'$. By Lemma~\ref{l:d_P-equality-for-N}, we can extend $f$ to some partial automorphism $\tilde{f}\colon B'\to C'$ preserving the condition $d_P(\tilde{f}(\bar{B}'), h_1(\bar{B}'))<\epsilon$. Now let $m=|B'|=|C'|$, and set
$$\eta_1=\frac{\epsilon-d_P(\tilde{f}(\bar{B}'),h_1(\bar{B}'))}{m2^{m-1}},$$
$$\eta_2=\min\{d(\tilde{f}(b),h_1(b)): b\in B'\setminus B,\tilde{f}(b)\neq h_1(b)\}.$$
Choose then $\delta>0$ with
$$\delta <\min\{\eta_1,\eta_2\},$$
or just $\delta<\eta_1$ if $\eta_2=0$.
By Lemma~\ref{l:small-moving}, there is $s\in H_C$ such that $0<d(s(c),c)<\delta$ for every $c\in C'\setminus C$. If we define $f'\colon B'\to s(C')$ as $f'=s\tilde{f}$, then:
\begin{itemize}[leftmargin=20pt]
\item $f'|_B=\tilde{f}|_B$, because $s\in H_C$;
\item we have:
\begin{align*}
d_P(f'(\bar{B}'),h_1(\bar{B}')) & \leq d_P(f'(\bar{B}'),\bar{C}')+d_P(\bar{C}',h_1(\bar{B}')) \\
& = d_P(s(\bar{C}'),\bar{C}')+d_P(\tilde{f}(\bar{B}'),h_1(\bar{B}')) \\
& < m2^{m-1}\eta_1 + d_P(\tilde{f}(\bar{B}'),h_1(\bar{B}')) = \epsilon,
\end{align*}
because $d_P(s(\bar{C}'),\bar{C}')\leq m2^{m-1}\delta$ (recall the inequality (\ref{eq:d_P-inequalities}) from the beginning of Subsection~\ref{ss:meas-alg}) and $\delta<\eta_1$;
\item $f'(b)\neq b$ for every $b\in B'\setminus A$, because $f'(b)=\tilde{f}(b)\neq b$ for $b\in B\setminus A$, and $0<d(s(c),c)<\eta_2$ for $c\in C'\setminus C$ (or simply $0<d(s(c),c)$, in the case $\eta_2=0$).
\end{itemize}
This proves the forward direction of the back-and-forth, and the other is symmetric. Since $F$ has the back-and-forth property, we conclude that there is $h_1'\in H$ with $h_1'|_A=h_1|_A$ and $\partial(h_1',h_1)\leq \epsilon$, and $h_1(b)\neq b$ for every $b\in N_0\setminus A$. We may finally define $\bar{h}'$ from $\bar{h}$ by replacing the first coordinate by $h_1'$ and leaving the rest unchanged. Note that by construction and since $\bar{h}$ (which coincides with $\bar{g}$ on $A$) does not fix any $a\in A\setminus\{{\bf 0},{\bf 1}\}$, the tuple $\bar{h}'$ has no (non-trivial) fixed point. Hence we have $\bar{h}'\in \bar{g}H_a^k\cap E_k$ and $\partial(\bar{h}',\bar{h})\leq\epsilon$, as desired.
\end{proof}

\begin{rem}\label{rem:Ek-unif-dense-in-G}
For the argument (and statement) of the proposition it is not necessary that $\bar{g}\in E_k$, but only that $\bar{g}$ have no fixed point on the algebra $A$ generated by $a$. As the proof shows, a tuple $\bar{g}\in H^k$ without fixed points on some finite subalgebra $A$ coincides on $A$ with some $\bar{g}'\in E_k$.
\end{rem}

\begin{cor}\label{c:eps-fatts-are-open-variant}
The $\epsilon$-fattening of any relatively open subset of $E_k$ is open in $G^k$.
\end{cor}
\begin{proof}
It suffices to check it for the base of open neighborhoods of $E_k$ given by the sets $\bar{g}H^k_a\cap E_k$ for $\bar{g}\in E_k$ and $a\in N_0^n$. By the previous proposition, $(\bar{g}H^k_a\cap E_k)_\epsilon = \bar{g}(G_a^k)_\epsilon$, which is open by Corollary~\ref{c:eps-fatts-are-open}.
\end{proof}

The next lemma should be seen as an ergodic variant of the Hrushovski property of $\MBAQ$.

\begin{lem}\label{l:ergodic-Hru-prop}
For every $\bar{g}\in E_k$ and finite subalgebra $A\subseteq N_0$, there exist $\bar{h}\in E_k$ and a finite subalgebra $B\subseteq N_0$ such that $A\subseteq B$, $\bar{g}|_A=\bar{h}|_A$, and $\bar{h}$ restricts to a tuple of automorphisms of $B$.
\end{lem}
\begin{proof}
By the Hrushovski property of $\MBAQ$ and ultrahomogeneity, there is $\bar{f}\in H^k$ extending $\bar{g}$ on $A$ that restricts to a tuple of automorphisms of some finite subalgebra $B\supseteq A\cup\bar{g}(A)$. We may moreover assume that all atoms of $B$ have equal measure. In particular, $\bar{f}$ has no non-trivial invariant points on $A$. We will show how to modify $\bar{f}$ on $B$, without changing its values on $A$, to obtain an ergodic action $\mbF_k\actson B$. By Remark~\ref{rem:Ek-unif-dense-in-G}, this is enough to obtain $\bar{h}$ as in the statement.

Let $B=\bigsqcup_{l=1}^m b^l$ be the decomposition of $B$ into minimal $\bar{f}$-invariant elements, and let $b^l=\bigsqcup_{j=1}^{m_l} b^l_j$ be the decomposition of each $b^l$ into atoms of $B$. For $j=1,\dots,m_1$, let $a_j$ be the unique atom of $A$ with $b_j^1\subseteq a_j$. Let also $a=\bigcup_{j=1}^{m_1}a_j$, so in particular $b^1\subseteq a$. If $f_i(a)\subseteq b^1$ for some $1\leq i\leq k$, then $\mu(b^1)\leq\mu(a)=\mu(f_i(a))\leq \mu(b^1)$, so actually $a=b^1=f_i(a)$. Since $A$ has no non-trivial $\bar{f}$-invariant elements, there exists $i$ such that $f_i(a)\nsubseteq b^1$. This means that there are $j_1\leq m_1$, $l>1$, and $j_2\leq m_l$ such that $b^l_{j_2}\subseteq f_i(a_{j_1})$.

Let $s\in H$ be an involution exchanging $f_i(b^1_{j_1})$ and $b^l_{j_2}$ and leaving all other atoms of $B$ fixed. Note that as $f_i(b^1_{j_1})\cup b^l_{j_2}\subseteq f_i(a_{j_1})$, the involution $s$ is the identity on the algebra $f_i(A)$. We then define $f_i'=sf_i$, and we let $\bar{f}'$ be the tuple obtained from $\bar{f}$ by replacing $f_i$ by $f_i'$ and leaving the other coordinates unchanged. In particular, $\bar{f}'|_A=\bar{f}|_A$. On the other hand, $b^1_{j_1}$ and $b^l_{j_2}$ are now in the same $f_i'$-orbit, so $b^1\cup b^l$ is a minimal $\bar{f}'$-invariant element of $B$. We thus see that after $m-1$ iterations of this construction we obtain an ergodic $\mbF_k$-action on $B$ extending the action of $\bar{g}$ on $A$.
\end{proof}

The proof of the following is now exactly as that of Theorem~\ref{thm:Fk-on-N-is-generic}, but easier.

\begin{theorem}\label{thm:Fk-on-N0-is-generic-in-Ek}
The action $\mbF_k\actson N_0$ is generic in $E_k$.
\end{theorem}
\begin{proof}
Let $\bar{p}\in H^k$ correspond to the generators of the action $\mbF_k\actson N_0$, and note that $H\cdot\bar{p}\subseteq E_k$. Lemma~\ref{l:ergodic-Hru-prop} and the fact that the action $\mbF_k\actson N_0$ contains a copy of every ergodic action on a finite algebra $A\in\MBAQ$, readily imply that $H\cdot\bar{p}$ is dense in $E_k$. For comeagerness it is then enough to see that for every finite quotient $\pi\colon \wFk\to \Gamma$ we have
$$\bar{p}H^k_{A_\pi}\cap E_k\subseteq\ov{H_{A_\pi}\cdot\bar{p}}.$$
To see this, in turn, and again by Lemma~\ref{l:ergodic-Hru-prop}, it suffices to show that for every finite subalgebra $C\subseteq N_0$ extending $A_\pi$ and every $\bar{f}\in E_k$ with $\bar{f}|_{A_\pi}=\bar{p}|_{A_\pi}$ that restricts to a tuple of automorphisms of $C$, there is $h\in H_{A_\pi}$ such that $\bar{f}|_C=(h\cdot\bar{p})|_C$. For such $C$ and $\bar{f}$, let $\Gamma'$ be the group generated by $\bar{f}$ within $\Aut(C)$ and $\pi'\colon\wFk\to\Gamma'$ be the induced quotient. Let also $P$ be the set of atoms of $C$. Choose $c_0\in P$ contained in $a_{\pi,1_\Gamma}$, and for any other $c\in P$ set
$$S_c\coloneqq \{\gamma'\in\Gamma':\gamma' c_0=c\}\text{ and }\sigma(c)\coloneqq a_{\pi',S_c}\in A_{\pi'}.$$
Since $\bar{f}'\in E_k$, the action $\Gamma'\actson P$ is transitive, and thus
$$\mu(c)=\frac{1}{|P|}=\frac{|S_c|}{|\Gamma'|}=\mu(\sigma(c)).$$
As in the proof of Theorem~\ref{thm:Fk-on-N-is-generic}, it follows that $\sigma$ induces an isomorphism $\sigma\colon C\to C'\subseteq N_0$ that fixes $A_\pi$ and conjugates $\bar{f}|_C$ to $\bar{p}|_{C'}$. By ultrahomogeneity, we get $h\in H_{A_\pi}$ with $\bar{f}|_C=(h\cdot\bar{p})|_C$, as desired.
\end{proof}

We can finally combine all the previous results and conclude:

\begin{theorem}\label{thm:metric-gen-profinite}
The pmp system $\mbF_k\actson \wFk$ is metrically generic, and in particular existentially closed.
\end{theorem}

The two models of \PMPFk* that we have exhibited, an ergodic and a non-ergodic one, are of course non-isomorphic, but as the uniform closures of their conjugacy classes are comeager, we have the following.

\begin{cor}\label{c:profinite-vs-profinite-times-trivial}
The pmp systems $\mbF_k\actson \wFk$ and $\mbF_k\actson \wFk\otimes \mbI$ are approximately isomorphic.
\end{cor}

\begin{question}\label{q:profinite-vs-profinite-times-trivial}
What are explicit $\epsilon$-isomorphisms between these two systems?
\end{question}

\noindent\hrulefill\smallskip

\noindent\textbf{Added in proof:} Goldbring, Seward, and Tucker-Drob have informed us of some results about existentially closed pmp ststems that they have obtained \cite{GSTD}.  They  show that if the group $\Gamma$ is a \emph{limit group}, then the continuous theory of pmp $\Gamma$-actions admits a model companion.  An explicit, ergodic-theoretic axiomatization of that model companion is given in the case that $\Gamma$ is a model of the first-order theory of the free group.  They also characterize those $\Gamma$ for which the pmp system given by the profinite completion is existentially closed.
\smallskip

\noindent\hrulefill

\bibliographystyle{amsalpha}
\bibliography{biblio}

\end{document}